\documentclass[reqno]{amsart}
\usepackage{amsmath,mathtools,amssymb}
\usepackage{color}
\usepackage{hyperref}
\mathtoolsset{showonlyrefs}

\newtheorem{lemma}{Lemma}
\newtheorem{theorem}[lemma]{Theorem}
\newtheorem{corollary}[lemma]{Corollary}
\newtheorem{proposition}[lemma]{Proposition}
\theoremstyle{definition}
\newtheorem{definition}{Definition}
\newtheorem{assumption}{Assumption}
\theoremstyle{remark}
\newtheorem{remark}{Remark}
\usepackage{hyperref}

\newcommand{\rd}{{\rm d}}
\newcommand{\e}{{\rm e}}

\newcommand{\N}{{\mathbb N}}
\newcommand{\R}{{\mathbb R}}
\newcommand{\C}{{\mathbb C}}
\newcommand{\Z}{{\mathbb Z}}
\newcommand{\eps}{\varepsilon}

\newcommand\re{\mathrm{Re}\,}
\newcommand\im{\mathrm{Im}\,}
\newcommand\I{\mathrm{i}}
\newcommand{\set}[2]{\{#1 : #2 \}}

\DeclareMathOperator{\Tr}{Tr}

\DeclareMathOperator{\supp}{supp}

\DeclareMathOperator{\Det}{det}

\begin{document}

\title[Schrödinger operators with complex sparse potentials]{Schrödinger operators with complex sparse potentials}
 \author[J.-C.\ Cuenin]{Jean-Claude Cuenin}
\address[J.-C.\ Cuenin]{Department of Mathematical Sciences, Loughborough University, Loughborough,
 Leicestershire, LE11 3TU United Kingdom}
 \email{J.Cuenin@lboro.ac.uk}

\date{\today}

\begin{abstract}
We establish quantitative upper and lower bounds for Schrödinger operators with complex potentials that satisfy some weak form of sparsity. Our first result is a quantitative version of an example, due to S.\ Boegli (Comm. Math. Phys., 2017, 352, 629-639), of a Schrödinger operator with eigenvalues accumulating to every point of the essential spectrum. The second result shows that the eigenvalue bounds of Frank (Bull. Lond. Math. Soc., 2011, 43, 745-750 and Trans. Amer. Math. Soc., 2018, 370, 219-240) can be improved for sparse potentials. The third result generalizes a theorem of Klaus (Ann. Inst. H. Poincaré Sect. A (N.S.), 1983, 38, 7-13) on the characterization of the essential spectrum to the multidimensional non-selfadjoint case. The fourth result shows that, in one dimension, the purely imaginary (non-sparse) step potential has unexpectedly many eigenvalues, comparable to the number of resonances. Our examples show that several known upper bounds are sharp.
\end{abstract}

\maketitle

\section{Introduction and main results}

\subsection{Introduction}
Many examples of Schrödinger operators with ``strange" spectral properties involve sparse potentials. In his seminal work \cite{MR484145} Pearson constructed examples of real-valued potentials (on the half-line) leading to singular continuous spectrum. 
The potentials consists of an infinite sequence of ``bumps" of identical profile, and the separation between these bumps increases rapidly. The physical interpretation is that a quantum mechanical particle will ultimately be reflected from a bump. These ideas were further developed in several directions, see e.g.\ \cite{MR1463044} \cite{MR1342046}, \cite{MR1628290}, \cite{MR1666767}, \cite{MR1866165}, \cite{MR2027640}, and the references therein. Scattering from sparse potentials in higher dimensions was studied by Molchanov and Vainberg \cite{MR1690097}, \cite{MR1756700}; see also \cite{MR1606716}, \cite{MR1803388},\cite{MR2724615}, \cite{MR3390425}, \cite{MR2073594}. The discrete spectrum for multidimensional lattice Schrödinger operators was investigated by Rozenblum and Solomyak \cite{MR2544038}. They constructed examples of sparse potentials whose number of negative eigenvalues grows like an arbitrary given polynomial power in the large coupling limit. 
In the recent work \cite{MR3627408} Boegli constructed a complex-valued sparse potential with arbitrary small $L^q$ norm ($q>d$) that has infinitely many non-real eigenvalues accumulating at every point of the essential spectrum.
Since the proof is based on compactness arguments, there is no quantitative bound on the rate of separation between the bumps, and hence no estimate on the pointwise decay of the potential or the accumulation rate of the eigenvalues is possible. 

\subsection{A quantitative version of Boegli's example}
Our first result provides quantitative decay bounds for the example in \cite{MR3627408}. Perhaps more importantly, the construction can be used to produce a potential together with an infinite number of eigenvalues (possibly not all of them) satisfying given upper and lower bounds on their accumulation rate. We formulate our result for the most interesting spectral region
\begin{align}\label{def. Sigma_epsilon_0}
\Sigma_{0}=\set{z\in\C}{|\im z|\leq \epsilon_0  \re z},
\end{align}
where $\epsilon_0>0$ is small but fixed.
\begin{theorem}\label{thm quant Boegli intro}
Let $d\geq 1$, $q>d$, $\epsilon_1,\epsilon_2\in (0,1],\gamma>0$, and let $(\zeta_n)_n\subset \Sigma_{0}$ be a sequence satisfying
\begin{align}\label{condition on sequence zetan intro}
\big(\sum_{n\in\N}|\zeta_n|^{\frac{d}{2}}|\im\zeta_n|^{q-d}|\log^{d}|\im\zeta_n/\zeta_n||\big)^{\frac{1}{q}}\asymp \epsilon_1.
\end{align}
Then there exists a complex sparse potential $V$ such that the following hold:
\begin{enumerate}
\item[a)] For each $n\in\N$ there exists a discrete eigenvalue $z_n$ of $H_V=-\Delta+V$ which is exponentially close to $\zeta_n$, in the sense $|z_n-\zeta_n|\leq \exp(-|\im\zeta_n|^{-\gamma})$.
\item[b)] The potential satisfies $\|V\|_{L^q(\R^d)}\lesssim  \epsilon_1$.
\item[c)] The potential decays polynomially, i.e.\ there exists a positive constant $\beta=\beta(\gamma,d,q)$ such that $|V(x)|\lesssim \epsilon_2\langle x\rangle^{-\beta}$.
\end{enumerate}
\end{theorem}
\begin{remark}
\noindent (i) In particular, for any $\lambda\in (0,\infty)$ there exists a sequence $(\zeta_n)_n\subset \Sigma_{0}$ satisfying \eqref{condition on sequence zetan intro} such that $\lim_{n\to\infty}\zeta_n=\lambda$. In this way one can find a sequence accumulating to every point of the essential spectrum. This yields a constructive proof of the result of Boegli \cite{MR3627408}.

\noindent (ii) One can remove the logarithm in \eqref{condition on sequence zetan intro} at the expense of replacing the $L^q$ norm of $V$ by the ``Davies--Nath norm" (see \eqref{def. Davies--Nath norm}).

\noindent (iii) We will give explicit bounds on the polynomial decay $\beta$.

\noindent (iv) Substituting the trivial lower bound $|\zeta_n|\geq |\im\zeta_n|$ into \eqref{condition on sequence zetan intro} shows $\im\zeta_n\to 0$. This is the reason why we say that $z_n$ is exponentially close to $\zeta_n$.
\end{remark}
We believe that the pointwise condition c) is more natural than the $L^q$ condition~b) for the phenomenon that takes place in Theorem \ref{thm quant Boegli intro}. This is because complex analogues of classical phase space bounds, which motivate the consideration on $L^q$ norms in the first place, lack many of the features that make them so useful for real potentials (more on that in Subsection \ref{subsection Weyl} below). Put simply, the $L^q$ norm does not see the separation between the bumps, while the pointwise bound does. We will nevertheless work with $L^q$ norms since we allow the bumps to have singularities. In the case where they are bounded the pointwise decay of the whole potential can easily be estimated by comparing the $L^{\infty}$ norms of the bumps to their spatial separation from the origin. In his fundamental work on non-selfadjoint Schrödinger operators, Pavlov \cite{MR0203530}, \cite{MR0234319} showed that the number of eigenvalues in one dimension is finite if $|V(x)|\lesssim \exp(-c|x|^{1/2})$, and that this exponential rate is best possible. This means that the potential in Theorem \ref{thm quant Boegli intro} cannot decay too fast. The $L^q$ bound imposes no decay whatsoever, but we can at least establish polynomial decay. For recent quantitative improvements of Pavlov's bound we refer to Borichev--Frank--Volberg~\cite{borichev2019counting} and Sodin \cite{MR4088346}.

The proof of the example in \cite{MR3627408} is based on ``soft" methods like weak convergence, compact embedding and the notion of the limiting essential spectrum. In contrast, our proof uses ``hard" estimates for the resolvent and the Birman-Schwinger operator, combined with tools from complex  analysis such as Rouch\'e's theorem, Jensen's formula and Cartan type estimates. This allows us to obtain more precise results thatn those in \cite{MR3627408}. Rouch\'e's theorem and Jensen's formula are among the most ubiquitous albeit simple tools in non-selfadjoint spectral theory, where such machinery as the variational principle or the spectral theorem is not available. In the present paper Cartan type estimates are crucial to bound a certain Fredholm determinant from below and get upper upper bounds on the norm of the resolvent. This opens the way to proving existence of eigenvalues by means of quasimode construction. We are then in a setting similar to the selfadjoint case where a quasimode of size $\epsilon$ guarantees the existence of a spectral point in an $\epsilon$-neighborhood of the quasi-eigenvalue. This follows from the inequality $\|(H_V-z)^{-1}\|\leq 1/\rd(z,\sigma(H_V))$, where $\sigma(H_V)$ is the spectrum. In the non-selfadjoint case the inequality may fail dramatically. This phenomenon gives rise to the notion of pseudospectrum, which we will not discuss here (see e.g.\ the monograph \cite{MR2359869}). The upper bounds obtained by Cartan type estimates generally grow exponentially in $1/\rd(z,\sigma(H_V))$. In order to beat this, we are forced to construct exponentially small quasimodes, a challenging task in all but the simplest models. The strategy is reminiscent of the proof of existence of resonances close to the real axis due to Tang--Zworski \cite{MR1637824} and Stefanov \cite{MR1700740} (see also the recent book by Dyatlov--Zworski \cite{MR3969938}). Our method is perhaps closest to that of Dencker--Sjöstrand--Zworski \cite[Section 6]{MR2020109} for non-selfajoint dissipative Schrödinger operators. The difference is that we consider decaying potentials and do not assume, as these authors do, that the quasi-eigenvalue is real (see \cite[Proposition 6.4]{MR2020109}). This means that the amplification of the exponential upper bound through the maximum principle (see \cite[Proposition 6.2]{MR2020109}) is in general not possible in our case. 
Another crucial difference is that we need a more quantitative version of the Cartan type estimate (Lemma \ref{lemma Levin}) as well as of the conformal transformations between the spectral region and the model domain (the unit disk). The Riemann mapping theorem is notoriously non-quantitative. Instead, we use Cayley and Schwarz--Christoffel transformations, which have previously been used in other contexts related to non-selfadjoint spectral theory, especially in connection with Lieb--Thirring type inequalities. The combination with Rouch\'e's theorem and the Cartan type bounds is new and leads to results with an inverse problem flavor, as in Theorem \ref{thm quant Boegli intro}.

\subsection{Magnitude bounds}
The second result gives precise bounds on the magnitude of eigenvalues of Schrödinger operators with complex sparse potentials, or more generally, potentials of the form $V=\sum_{j=1}^NV_j$, where the $V_j$ have disjoint support and separate rapidly from each other. We will call these ``separating" potentials. The Schrödinger operator $H_V=-\Delta+V$ behaves like an almost orthogonal sum, due to the rapid decoupling between the $N$ ``channels". This enables us to improve upon the bounds for general complex potentials due to Frank \cite{MR2820160}, \cite{MR3717979}. For simplicity we state the result here for $d\geq 3$. The general case along with further refinements can be found in Subsection \ref{subsection Magnitude bounds}. We refer to Section \ref{section Definitions} for a more in-depth explanation of the terminology.

\begin{theorem}\label{theorem sparse magnitude bound intro}
Assume that $d\geq 3$ and $d/2\leq q\leq (d+1)/2$. If $V$ is separating at scale $\eta^{-1}$, then every eigenvalue $z$ of $H_V$ with $\im\sqrt{z}\geq (d+1)\eta$ satisfies
\begin{align}\label{sparse magnitude bound}
|z|^{q-\frac{d}{2}}\lesssim  \sup_{j\in[N]}\|V_j\|_{L^q(\R^d)}^q.
\end{align}
If $q>(d+1)/2$, then
\begin{align}\label{sparse magnitude bound q>(d+1)/2}
|z|^{\frac{1}{2}}\rd(z,\R_+)^{q-\frac{d+1}{2}}\lesssim  \sup_{j\in[N]}\|V_j\|_{L^q(\R^d)}^q.
\end{align}
\end{theorem}
The bound \eqref{sparse magnitude bound} follows from \eqref{BS bound} by a Birman--Schwinger argument. It could also be proved by using the eigenvalue bounds of the author \cite{MR4104544}, which are inspired by a method of Davies and Nath \cite{MR1946184} in one dimension. For $N=1$ the estimates \eqref{sparse magnitude bound}, \eqref{sparse magnitude bound q>(d+1)/2} coincide with those of Frank \cite{MR2820160}, \cite{MR3717979}, respectively. The difference is that here $V$ might decay very slowly or not at all. Nevertheless, on the $\eta^{2}$ energy (spectral) scale the estimate is of the same quality as for $N=1$.

We make a short remark about the connection with the Laptev--Safronov conjecture \cite{MR2540070}, which stipulates that
\begin{align}\label{LS conjecture}
\sup_{V\in L^q(\R^d)}\sup_{z\in\sigma(-\Delta+V)\setminus\R_+}\frac{|z|^{q-\frac{d}{2}}}{\|V\|_q^q}<\infty\quad\mbox{for all}\quad q\in[d/2,d].
\end{align}
For the range $q\in[d/2,(d+1)/2]$ the conjecture was proven by Frank \cite{MR2820160}. The question whether \eqref{LS conjecture} is true for $q\in((d+1)/2,d]$ is still open. 
The expectation, based on intuition from counterexamples to Fourier restriction (see e.g.\ \cite{MR4107520}, \cite{MR4104544} for more explanations) is that the conjecture is false in this range. Incidentally, Theorem \ref{thm quant Boegli intro} clearly implies the necessity of $q>d$ in the conjecture, but this already follows from B\"ogli's result (without the pointwise bound). In fact, a single bump of the sparse potential used in Boegli's construction (and in Theorem \ref{thm quant Boegli intro}) already provides a counterexample. Since there seems to be some confusion about this issue we use the results of Section \ref{section Complex square wells} to show necessity of the condition $q>d$. Indeed, Lemma \ref{lemma single square well} implies that for small $\epsilon>0$ there is a potential $V(\epsilon)$ and an eigenvalue $z(\epsilon)$ such that
$|z(\epsilon)|^{q-\frac{d}{2}}/\|V(\epsilon)\|_q^q\gtrsim \epsilon^{q-d}\log^q(1/\epsilon)$.
The example (a complex step potential) is simple but, quite amazingly, generic enough to show optimality of several estimates in the literature (see \cite{MR4104544}). In one dimension, the step potential can be tuned to essentially saturate any of the known magnitude bounds. For example, the last inequality also shows that the Davies--Nath bound \cite{MR1946184} is sharp and, since $\im z(\epsilon)\asymp\epsilon$, that Frank's bound \cite[Theorem 1.1]{MR3717979}
is sharp up to a logarithm. In Subsection \ref{subsection Weyl} we will show how the complex step potential also implies sharpness of another bound in \cite{MR3717979}.

\subsection{A generalization of Klaus' theorem}

The following is a generalization of a result due to Klaus \cite{MR700696} on the characterization of the essential spectrum. The generalization is twofold: First, we admit complex potentials and second, we prove it for any dimension (whereas Klaus only proved the one-dimensional case). In this introduction we again focus on the case $d\geq 3$, but the statement is valid in $d=1,2$ for $q$ in the range \eqref{range q}.

\begin{theorem}[Klaus' theorem \cite{MR700696} for complex potentials]\label{thm. Klaus' theorem intro}
Assume that $d\geq 3$ and $d/2\leq q\leq (d+1)/2$. If $V$ is a separating potential and $\sup_{j\in[N]}\|V_j\|_{L^q(\R^d)}<\infty$, then
\begin{align*}
\sigma_{\rm e}(H)=[0,\infty)\cup S,
\end{align*}
where $S$ is the set of all $z\in\C\setminus [0,\infty)$ such that there exist infinite sequences $(i_n)_n,(z_n)_n$ with $z_n\in \sigma(H_{V_{i_n}})$, $i_n\to\infty$ and $z_n\to z$ as $n\to\infty$.
\end{theorem}
An alternative proof (also in one dimension) of Klaus' theorem can be found in~\cite{MR883643}. The role of Theorem \ref{thm. Klaus' theorem intro} in this paper will be an auxiliary one, and we will only use it to argue that the essential spectrum is invariant under the perturbations we consider in Theorem \ref{thm quant Boegli intro}. Although our proof follows the general strategy of \cite{MR700696} it is still worth emphasizing that some parts of it require somewhat novel techniques.

\subsection{Weyl's law and locality}\label{subsection Weyl}

Recently, Boegli and Štampach \cite{boegli2020liebthirring} disproved a conjecture by Demuth, Hansmann and Katriel \cite{MR3016473} for one-dimensional Schrödinger operators with complex potentials by establishing a lower bound on a certain Riesz means of eigenvalues. 
More precisely, consider $H_{\alpha V}$,
where $V=\I\mathbf{1}_{[-1,1]}$ is a purely imaginary step potential and $\alpha$ is a large semiclassical parameter. Boegli and Štampach prove that, for any $p\geq 1$,
\begin{align}\label{main result SF}
\alpha^{-p}\sum_{z\in\sigma_d(H_{\alpha V})}\frac{(\im z)^p}{|z|^{1/2}}\geq C_p\log\alpha.
\end{align} 
The interesting feature of this bound is that it shows a logarithmic violation of Weyl's law. To recall Weyl's law, consider a \emph{self-adjoint} operator, with a smooth real-valued potential. Note that if we set $h=1/\sqrt{\alpha}$, then $\alpha^{-1}H_{\alpha V}$ takes the form of a semiclassical Schrödinger operator, $-h^2\partial_x^2+V(x)$. Semiclassical asymptotics (Weyl's law) yield, for a suitable class of functions $f$,
\begin{align}\label{Weyl law f}
\Tr f(H_{\alpha V})=\frac{\sqrt{\alpha}}{2\pi}\big(\int f(\alpha(\xi^2+V(x)))\rd x\rd \xi+o(1)\big)
\end{align}
as $\alpha\to \infty$. In particular, if $f$ is homogeneous of degree $\gamma$, then 
\begin{align}\label{Weyl law}
\lim_{\alpha\to \infty}\alpha^{-1/2-\gamma}\Tr f(H_{\alpha V})=(2\pi)^{-1}\int f(\xi^2+V(x))\rd x\rd \xi.
\end{align}
For $f(\lambda):=\lambda_-^{\gamma}$ and $\gamma\geq 1/2$ (since we are considering $d=1$) the Lieb-Thirring inequality
\begin{align}\label{LT}
\sum_{\lambda\in\sigma_{\rm d}(H_{\alpha V})}\lambda_-^{\gamma}\leq C_{\gamma}\alpha^{1/2+\gamma}\int V_-(x)^{1/2+\gamma}\rd x
\end{align}
captures the semiclassical behavior \eqref{Weyl law}, but is valid for any $\alpha>0$, not only asymptotically. Returning to the complex potential $V=\I\mathbf{1}_{[-1,1]}$ and noticing that $f(z):=\frac{(\im z)^p}{|z|^{1/2}}$ is homogeneous of degree $\gamma=p-1/2$, we observe that \eqref{main result SF} implies that the formal analogue of \eqref{Weyl law} cannot hold, i.e.\ that
\begin{align*}
\liminf_{\alpha\to\infty}\alpha^{-1/2-\gamma}\sum_{z\in\sigma_d(H_{\alpha V})}\frac{(\im z)^{1/2+\gamma}}{|z|^{1/2}}=\infty,
\end{align*} 
hence violating Weyl's law \eqref{Weyl law}. The comparison with Weyl's law is formal because $f(H_{\alpha})$ does not make sense in general for a non-normal operator. However, \eqref{main result SF} also shows that the complex analogue of the Lieb-Thirring inequality \eqref{LT},
\begin{align}
\sum_{z\in\sigma_d(H_{\alpha V})}\frac{(\im z)^{1/2+\gamma}}{|z|^{1/2}}\leq C_{\gamma}'\alpha^{1/2+\gamma}\int V_-(x)^{1/2+\gamma}\rd x,
\end{align}
cannot be true, thus disproving the conjectured bound in \cite{MR3016473}. 

For a non-selfadjoint (pseudo)-differential operator with analytic symbol and in one dimension the eigenvalues typically lie on a complex curve, hence violating Weyl's law (in terms of complex phase space). Small random perturbations typically restore the Weyl law (in real phase space). The literature on the subject is vast and we merely refer the interested reader to the book of Sjöstrand \cite{MR3931708} for an overview of recent developments. In contrast, a classical result of Markus and Macaev \cite{MR545380} implies that Weyl's law holds for the real part of the eigenvalues, provided the non-selfadjoint perturbation is small.

A slightly different view on the same phenomenon (violation of Weyl's law) is connected with the notion of ``locality". In the semiclassical limit, in the self-adjoint case, each state occupies a volume of $(2\pi/\sqrt{\alpha})$ in phase space $T^*\R$.
Indeed, with $f(\lambda):=\mathbf{1}_{(-\infty,-E]}(H_{\alpha V})$, \eqref{Weyl law f} yields
\begin{align*}
N(H_{\alpha V};(-\infty,-E])=\sqrt{\alpha}\,\mathrm{Vol}(S_{E,V})(1+o(1)),
\end{align*}
where $N(H;\Sigma)$ denotes the number of eigenvalues of an operator $H$ in a set $\Sigma$ and $S_{E,V}:=\set{(x,\xi)\in T^*\R}{\xi^2+V(x)\leq -E}$ is the relevant part of phase space. If we consider a sum of disjoint bumps $V=\sum_{j=1}^N V_j$, say with $V_j(x)=W(x-x_j)$, then $S_{E,V}=N\,S_{E,W}$, and hence 
\begin{align*}
\frac{N(H_{\alpha V};(-\infty,-E])}{N(H_{\alpha W};(-\infty,-E])}=N(1+o(1)).
\end{align*}
This means that in the semiclassical limit each bump is responsible for an equal number of eigenvalues. In particular, two distinct realizations of $V$ as a sum of bumps have the same number of eigenvalues, as long as the bumps are disjoint. This feature of locality is also captured by the Lieb-Thirring inequalities since the bound involves the integral linearly.

Our third result shows that this kind of locality can be violated in the non-selfadjoint case. We adopt the same notation $N(H;\Sigma)$ for the number of eigenvalues of $H$ in $\Sigma$ as in the self-adjoint case, but emphasize that these are counted according to their \emph{algebraic} multiplicity.
We consider one sparse one non-sparse (or non-separating) realization of $V$ and denote these by $V_{\rm s}$ and $V_{\rm n}$, respectively. For simplicity, we will consider the same potential as in \cite{boegli2020liebthirring}, i.e.\ the  purely imaginary step potential $W=\I|W_0|\mathbf{1}_{[-R_0,R_0]}$ of size $|W_0|$ and width $R_0$. For simplicity we fix these scales  to be of order one. 
Then $V_{\rm n}$ is the single well of width $R=NR_0$, while $V_{\rm s}$ is a sum of $N$ disjoint wells $W(x-x_j)$ of width $R_0$. We will fix the coupling strength $\alpha$ and consider the limit $N\to\infty$. For the non-sparse operator this is in fact still a semiclassical limit, as can easily be seen by rescaling.

\begin{theorem}\label{thm. violation of locality into}
Let $d=1$, $N\gg 1$ and consider the rectangular set
\begin{align*}
\Sigma:=\set{z\in\C}{C^{-1}\frac{N^2}{\log^2 N}\leq \re z\leq C\frac{N^2}{\log^2 N},\, C^{-1}\leq \im z\leq C},
\end{align*}
where $C$ is a large constant.
Then we have
\begin{align}\label{Vnsp}
N(H_{V_{\rm n}};\Sigma)\gtrsim \frac{N^2}{\log N}.
\end{align}
Moreover, there exists a sequence $(x_j)_j$ such that
\begin{align}\label{Vsp}
N(H_{V_{\rm s}};\Sigma)\lesssim N.
\end{align}
\end{theorem}
The estimates in \cite{boegli2020liebthirring} would be sufficient to prove a lower bound of size $N^{2-\epsilon}$ in \eqref{Vnsp}. Their argument proceeds by approximating the characteristic eigenvalue equation and finding the roots of this equation in an asymptotic regime.
They did not prove that the original equation has nearby roots. This can be done e.g.\ by a contraction mapping argument \cite{stepanenko}. We will give a proof using Rouch\'e's theorem. Note that, by power counting (dimensional analysis), the constants in \eqref{Vnsp}, \eqref{Vsp} only depend on the dimensionless quantity $|W_0|^{1/2}R_0$.

In the selfadjoint case, i.e.\ when $W$ is replaced by $|W_0|\mathbf{1}_{[-R_0,R_0]}$, the number of negative eigenvalues of $V_{\rm n}$ is of order~$N$, in agreement with semiclassics. Also note that, since in one dimension each $H_{V_j}$ has at least one negative eigenvalue, we have $N(H_{V_{\rm s}};\R_-)\asymp N$ in this case. Hence the left hand sides of \eqref{Vnsp} and \eqref{Vsp} are equal in magnitude, which may be seen as a manifestation of locality. However, this locality is violated if one takes into account not only eigenvalues but also resonances. Zworski \cite{MR899652} proved that, for a compactly supported, bounded, complex potential, the number $n(r)$ of resonances $\lambda_j^2$ in a disk $|\lambda_j|\leq r$ asymptotically satisfies
\begin{align}\label{Zworski}
n(r)=\frac{2\,\mathrm{|ch\, supp(V)}|}{\pi}\,r(1+o_V(1))
\end{align}
as $r\to\infty$. Moreover, for any $\epsilon_0>0$, the number of resonances in $|\lambda_j|\leq r$ but outside the sector $\Sigma_{0}$ (see \eqref{def. Sigma_epsilon_0}) is $o(r)$. This and the fact that eigenvalue bounds outside $\Sigma_{0}$ are ``trivial" (in the sense that they can be proved by the same standard estimates as for real potentials, with the provisio that the constants blow up as the implicit constant in \eqref{def. Sigma_epsilon_0} becomes small) motivates us to often restrict attention to the spectral set $\Sigma_{0}$. 
The result \eqref{Zworski} was obtained earlier by Regge \cite{MR143532} in some special cases. Different proofs were given by Froese \cite{MR1456597} and Simon \cite{MR1802901}. Froese's proof also works for complex potentials. 
Note that eigenvalues are included in the definition of resonances. Formula \eqref{Zworski} can be seen as a Weyl law for resonances but is \emph{nonlocal} as it includes the convex hull of the support of the potential. An obvious corollary of Zworski's formula is that the right hand side of \eqref{Zworski} is an upper bound for the number of eigenvalues $\lambda_j^2$ in the disk $|\lambda_j|\leq r$ (resonances in the upper half plane). For the potential $V_{\rm n}$, taking $r\asymp N/\log N$ and observing that $\mathrm{ch\, supp(V_{\rm n})}\asymp N$, we find that the leading term in \eqref{Zworski} is of order $N^2/\log N$, in agreement with the lower bound \eqref{Vnsp}. Note, however, that the asymptotics are not uniform, i.e.\ the error may depend on $N$ (recall that $V_{\rm n}$ depends on $N$). Korotyaev \cite[Theorem 1.1]{MR3574651} proved an upper bound which yields
\begin{align*}
n(r)\leq \frac{8N}{\pi\log 2}r+\mathcal{O}(N\log N)
\end{align*}
for $r\asymp N/\log N$. Hence, \eqref{Vnsp} shows that, at the scale considered here, \textit{a substantial fraction of resonances are actual eigenvalues}. In the special case of a step potential  Stepanenko \cite{stepanenko2020bounds} proved that the total number of eigenvalues is bounded by $N^2/\log N$. The lower bound \eqref{Vnsp} shows that this is sharp up to constants; this was observed independently by Stepanenko \cite{stepanenko}.

In dimension $d\geq 3$ it may of course happen that all the $H_{V_j}$, and thus also~$H_{V_{\rm s}}$, have no eigenvalues at all if $|W_0|^{1/2}R_0$ is small (by the CLR bound), while~$H_{V_{\rm n}}$ has of the order $N^d$ eigenvalues. This is not the kind of phenomenon that takes place in Theorem \ref{thm. violation of locality into}. Indeed, there we allow $|W_0|^{1/2}R_0$ to be of unit size. We also note that in higher dimensions the results on the asymptotics of the resonance counting function are weaker, and there are example of complex potentials with no resonances at all (see Christiansen \cite{MR2225694}). On the other hand, Christiansen \cite{MR2174422} and Christiansen--Hislop \cite{MR2189242}, \cite{MR2648080} proved that $n(r)$ has maximal growth rate $r^d$ for ``most" potentials in certain families.

A final comment regarding the implications of Theorem \ref{thm. violation of locality into}, also related to locality (or the lack thereof), concerns a general observation on Lieb--Thirring type inequalities for complex potentials. It is a fact that all known upper bounds for eigenvalue sums have a superlinear dependence on $N$. For example, in $d=1$, \cite[Theorem 1.3]{MR3717979} yields the bound
\begin{align}\label{Frank d=1 sum}
\sum_j\rd(z_j,\R_+)^a\lesssim \left(\int_{\R}|V|^b\rd x\right)^{c}
\end{align}
with $(a,b,c)=(1,1,2)$. The lower bound \eqref{Vnsp} shows that the power $c$ cannot be decreased, while preserving the homogeneity condition $-2a=(-2b+1)c$. Indeed, 
\begin{align*}
\sum_j\rd(z_j,\R_+)^a\geq \sum_{z_j\in\Sigma}\rd(z_j,\R_+)^a\gtrsim \frac{|W_0|^aN^2}{\log N},
\end{align*}
while
\begin{align*}
\left(\int_{\R}|V_{\rm n}|^{b}\rd x\right)^{c}=(2NR_0|W_0|^{b})^{c}, 
\end{align*}
so that the ratio of the first to the second is bounded below by $(|W_0|^{1/2}R_0)^{-c}\frac{N^{2-c}}{\log N}$, which tends to infinity as $N\to\infty$, for every $c<2$. On the other hand, for sufficient rapid separation of the bumps in $V_{\rm s}$, we can prove Frank's bound \eqref{Frank d=1 sum} with a linear dependence on $N$, at least locally in the spectrum.

\begin{theorem}\label{thm. Frank's 1d bound for sparse intro} Let $V\in \ell^2L^1(\R)$ and $N\gg 1$.
For any $\eta>0$ there exists a sequence $(x_j)_{j=1}^N$, $x_j=x_j(\eta,\|V\|_{\ell^2L^1})$, such that
\begin{align}\label{N times Frank's bound}
\sum_{\im\sqrt{z_j}>(d+1)\eta}\rd(z_j,\R_+)\lesssim N\sum_{j=1}^N\left(\int_{\R}|V_{j}|\rd x\right)^{2}
\end{align}
\end{theorem}

Our methods could be used to prove similar generalizations of the higher-dimensional results of \cite{MR3717979}, but we will not pursue this.

\subsection*{Notation} 
We write $\sigma(H)$, $\rho(H)$ for the spectrum, respectively the resolvent set of a closed linear operator $H$. The free and the perturbed resolvent operators are denoted by $R_0(z)=(-\Delta-z)^{-1}$ and $R_V(z)=(H_V-z)^{-1}$, respectively. We denote by $\mathfrak{S}^p$ the Schatten spaces of order $p$ over the Hilbert space $L^2(\R^d)$ and by $\|\cdot\|_p$ the corresponding Schatten norms. We also write $\|\cdot\|=\|\cdot\|_{\infty}$ for the operator norm. To distinguish it from $L^p$ norms of functions we denote the latter by $\|\cdot\|_{L^p}$. We will use the notation $V^{1/2}=V/|V|^{1/2}$ and $\langle x\rangle:=2+|x|$. The statement $a\lesssim b$ means that $|a|\leq C |b|$ for some absolute constant $C$. We write $a\asymp b$ if $a\lesssim b\lesssim a$. If the estimate depends on a list of parameters $\tau$, we indicate this by writing $a\lesssim_{\tau} b$. The dependence on the dimension $d$ and the Lebesgue exponents $p,q$ is usually suppressed. We write $a\ll b$ if $|a|\leq c|b|$ with a small absolute constant $c$, independent of any parameters. By an absolute constant we always understand a dimensionless constant $C=C(d,p,q)$. Here we choose units of length $l$ such that position, momentum and energy have dimensions $l$, $l^{-1}$ and $l^{-2}$, respectively. We chose the branch of the square root $\sqrt{\cdot}$ on $\C\setminus[0,\infty)$ such that $\sqrt{z}\in\mathbb{H}$, where $\mathbb{H}=\set{\kappa\in\C}{\im\kappa>0}$ denotes the upper half plane. The open unit disk in $\C$ is denoted by $\mathbb{D}$.

\subsection*{Acknowledgements}
The author gratefully acknowledges correspondence with Sabine Boegli and comments of Rupert Frank, who pointed out the failure of Weyl's law and the connection with nonlocality. Many thanks also go to St\'ephane Nonnenmacher for useful discussions on resonances and to Alexei Stepanenko for explaining his recent preprint. Special thanks go to Tanya Christiansen for many helpful remarks on a preliminary version of the introduction.

\section{Definitions and preliminaries}\label{section Definitions}

\subsection{Separating and sparse potentials}\label{subsection Separating and sparse potentials}
We consider sparse potentials of the form
\begin{align}\label{separating potential}
V(x)=V_j(x),\quad x\in \Omega_j,
\end{align}
where $j\in [N]=\{1,2,\ldots,N\}$, $N\in \N\cup\{\infty\}$, and $\Omega_j\subset\R^d$ are mutually disjoint (not necessarily bounded) sets. We then set
\begin{align}
L_j:=\rd(\Omega_j,\cup_{i\in[N]}\Omega_i\setminus\Omega_j).
\end{align}
We assume that $V_j\in \ell^{p}L^q$, where the norms are defined by 
\begin{align*}
\|V\|_{\ell^{p}L^q}:=\big(\sum_{j\in[N]}\|V_j\|_{L^{q}(\R^d)}^p\big)^{1/p}
\end{align*}
for $p\in [1,\infty)$ and $\|V\|_{\ell^{\infty}L^q}:=\sup_{j\in [N]}\|V_j\|_{L^{q}(\R^d)}$. Here $q$ will be in the range 
\begin{align}\label{range q}
q\in [1,\infty] \quad \mbox{if } d=1,\quad q\in (1,\infty]\quad \mbox{if } d=2,\quad q\in[\tfrac{d}{2},\infty] \quad \mbox{if } d\geq 3.
\end{align}
In particular, we have $\|V\|_{L^q}=\|V\|_{\ell^q L^q}$. We sometimes write $V=V(L)$ or $V=V(L,\Omega)$ to emphasize the dependence of $V$ on the sequences $L=(L_j)_{j=1}^N$ or $\Omega=(\Omega)_{j=1}^N$. 
\begin{definition}\label{def. separating}
We say that $V=V(L)$ is \emph{separating} if 
\begin{align*}
\mathrm{sep}(L,\eta):=\sum_{j\in [N]}\exp(-\eta L_j)<\infty
\end{align*}
for every $\eta>0$. We call $V$ \emph{separating at scale} $\eta^{-1}$ if $\mathrm{sep}(L,\eta)\leq 1$. We say that $V$ is \emph{strongly separating} if $\mathrm{sep}(L,\delta\eta)\lesssim_{\delta}\mathrm{sep}(L,\eta)$ for every $\delta,\eta>0$.
\end{definition}
The constant $\mathrm{sep}(L,\eta)$ only depends on $L$ and not on $V$ itself. We will sometimes abuse terminology and call the sequence $L$ separating. Note that since $\eta L_j$ is dimensionless, $\eta$ has the dimension of inverse length.
We shall always assume that the sequence $L$ is increasing.

\begin{definition}\label{def. sparse}
We say that $V=V(L,\Omega)$ is \emph{sparse} if it is separating, the supports $\Omega_n$ are bounded,
and $\lim_{n\to\infty}\mathrm{diam}(\Omega_n)/L_n=0$.
\end{definition}
Most of our results hold for separating potentials. The strong separation condition is convenient and facilitates some of the proofs.
Typical examples of strongly separating sequences are ($\eta\ll\eta_0$):
\begin{itemize}
\item[a)] It $\eta_0L_k\gtrsim k^{\alpha}$ for $\alpha>0$, then $\mathrm{sep}(L,\eta)\lesssim (\eta/\eta_0)^{-1/\alpha}$.
\item[b)] If $\eta_0L_k\gtrsim\exp(k)$, then $\mathrm{sep}(L,\eta)\lesssim \log(\eta_0/\eta)$.
\item[c)] If $\eta_0L_k\gtrsim\exp(\exp(k))$, then $\mathrm{sep}(L,\eta)\lesssim\log\log(\eta_0/\eta)$.
\end{itemize}
See Subsection \ref{subsection Distribution function} for a proof.
The explicit example used to prove Theorem \ref{thm quant Boegli intro} turns out to be sparse. Note that, by the disjoint supports, the definition \eqref{separating potential} is equivalent to $V=\sum_{j=1}^NV_j$. 

\subsection{Comparison with a direct sum}
In Section \ref{Section Comparison between Hdiag and HV} we will compare the two operators
\begin{align}\label{compare the two operators}
H_V=-\Delta+V,\quad
H_{\rm diag}=\bigoplus_{j\in[N]} (-\Delta+V_j).
\end{align}
Note that the point spectrum (eigenvalues) of $H_{\rm diag}$,
\begin{align}\label{spectrum direct sum}
\sigma_{\rm p} (H_{\rm diag})=\bigcup_{j=1}^N\sigma_{\rm p} (H_{V_j})=\bigcup_{j=1}^N\sigma_{\rm p} (H_{W_j}),
\end{align} 
is independent of the sequence $L$. We will consider $\sigma_{\rm p} (H_{W_j})$ as part of the data and seek to prove lower bounds on $L_j$ or $|x_i-x_j|$ that imply that $\sigma_{\rm p} (H_{V})$ is close to $\sigma_{\rm p} (H_{\rm diag})$. In fact, we will consider a subset of the point spectrum, the discrete spectrum. We will consider the $V_j$ as given only \emph{up to translations}, i.e.\ we stipulate that 
\begin{align}\label{def. Wj}
V_j(x)=W_j(x-x_j),
\end{align}
where $W_j\in \ell^{p}L^q$ contains the origin in its support. 
Using the triangle inequality, it is easy to see that $|x_i-x_j|\geq L_j$ for $i\neq j$, and therefore
\begin{align*}
\sup_{i\in[N]}\sum_{j\in [N]\setminus\{i\}}\exp(-\eta |x_i-x_j|)\leq \mathrm{sep}(L,\eta).
\end{align*}
Straightforward arguments also show that
\begin{align}\label{lower bound Li in terms of xi, Omegai}
L_i\geq |x_i-x_j|-\mathrm{diam}(\Omega_i)-\mathrm{diam}(\Omega_j)
\end{align}
for all $j\in [N]\setminus\{i\}$. In particular, for sparse potentials,
\begin{align}\label{lower bound Li in terms of xi, Omegai sparse}
L_i(1+o(1))\geq \sup_{j<i}|x_i-x_j|.
\end{align}
Note that $\mathrm{diam}(\Omega_j)=\mathrm{diam}(\supp W_j)$ is part of the data. We can thus obtain a lower bound for $|x_i-x_j|$ in terms of $L_j$ and vice versa. For this reason we restrict attention to $L_j$ here. 

\subsection{Truncations}
For technical reaons, it will turn out to be convenient to consider finite truncations. For $n\in[N]$ we define
\begin{align}
V^{(n)}&:=\sum_{j\in[n]} V_j,\quad
H^{(n)}:=-\Delta+V^{(n)}.\label{def. Vuppern}
\end{align}

\subsection{Abstract Birman-Schwinger principle}
We mostly disregard operator theoretic discussions here and refer e.g.\ to \cite{MR3717979} for the (standard) definition of $H_{V_j}$ as $m$-sectorial operators. The rigorous definition of $H_V$ is a bit more subtle since $V$ need not be decaying. However, a classical construction of Kato \cite{MR0190801} produces a closed extension $H_V$ of $-\Delta+V$ via a Birman-Schwinger type argument. This approach works as soon as one can find point $z_0$ in the resolvent set of $H_0=-\Delta$ at which the Birman-Schwinger  operator
\begin{align}\label{def. BS_V(z)}
BS_V(z):=|V|^{1/2}(-\Delta^2-z)^{-1}V^{1/2}
\end{align}
has norm less than one. Such bounds are provided by Lemma \ref{lemma BS bounds}, but to avoid technicalities it is useful to think of the potential as being bounded by a large cutoff (and all of the bounds will be independent of that cutoff). 
By iterating the second resolvent identity,
\begin{align*}
R_V(z)=R_0(z)-R_0(z)VR(z), 
\end{align*}
it is then easy to see that
\begin{align}\label{resolvent identity}
R_V(z)-R_0(z)=-R_0(z)V^{1/2}(I+BS_V(z))^{-1}|V|^{1/2}R_0(z).
\end{align}
For more background on the abstract Birman-Schwinger principle in a nonselfadjoint setting we refer to \cite{MR2201310}, \cite{MR3717979}, \cite{behrndt2020generalized}, \cite{hansmann2020abstract}.

\subsection{The essential and discrete spectrum}
We briefly recall some facts about the essential and discrete spectrum of a closed operator $H$. There are several inequivalent definitions of essential spectrum for non-selfadjoint operators (but these all coincide for Schrödinger operators with decaying potentials \cite[Appendix B]{MR3717979}). We use the following standard definition.

\begin{align*}
\sigma_{\rm e}(H):=\set{z\in\C}{H-z \mbox{ is not a Fredholm operator}}.
\end{align*}

The discrete spectrum is defined as
\[
\sigma_{\rm d}(H):=\set{z\in\C}{z\mbox{ is an isolated eigenvalue of $H$ of finite multiplicity}}.
\]
Note that, if $H$ is not selfadjoint, then, in general, $\sigma(H)$ is not the disjoint union of $\sigma_{\rm e}(H)$ and $\sigma_{\rm d}(H)$. However, by \cite[Theorem XVII.2.1]{GGK1}, if every connected component of $\C\setminus\sigma_{\rm e}(H)$ contains points of $\rho(H)$, then 
\begin{align}\label{sigmae,sigmad}
\sigma(H)\setminus\sigma_{\rm e}(H)=\sigma_{\rm d}(H).
\end{align}

In the situations we consider here \eqref{sigmae,sigmad} will always be true for $H=H_V$ and $H_{\rm diag}$. In fact, Corollary \ref{cor. essential spectrum} tells us that 
$\sigma_{\rm e}(H)=[0,\infty)$, just as for decaying potentials.

\section{Universal bounds for separating potentials}
\label{section Quantitative bounds for separating potentials}
In this section we consider $H_V$ as a perturbation of $-\Delta$. We will thus only make assumptions about $V$ and not about $H_{\rm diag}$.

\subsection{Birman--Schwinger analysis}
Since the $V_j$ have mutually disjoint supports, we can write the Birman-Schwinger operator \eqref{def. BS_V(z)} as
\begin{align*}
BS_V(z)=\sum_{i=1}^NBS_{ii}(z)+\sum_{i\neq j}BS_{ij}(z),
\end{align*}
where 
\begin{align*}
BS_{ij}(z)=|V_i|^{1/2}R_0(z)V_j^{1/2}.
\end{align*}
The first term is unitarily equivalent to the orthogonal sum 
\begin{align*}
BS_{\rm diag}(z):=\bigoplus_{i=1}^NBS_{ii}(z).
\end{align*}
on the Hilbert space $\mathcal{H}\simeq\bigoplus_{i=1}^N \mathcal{H}_i$, where $\mathcal{H}=L^2(\R^d)$, $\mathcal{H}_i=L^2(\Omega_i)$. By abuse of notation we will always identify these two Hilbert spaces and the corresponding operators. The off-diagonal contribution is
\begin{align*}
BS_{\rm off}(z):=\sum_{i\neq j}BS_{ij}(z).
\end{align*}

In the following we will use the notation
\begin{equation}\label{def omegap(z)}
\begin{split}
\omega_q(z):=\begin{cases}
|z|^{\frac{d}{2q}-1}\quad&\mbox{if } q\leq q_d,\\
|z|^{-\frac{1}{2q}}\rd(z,\R_+)^{\frac{q_d}{q}-1}\quad&\mbox{if } q\geq q_d,
\end{cases}
\end{split}
\end{equation}
where $q_d=(d+1)/2$. Note that for $z\in\Sigma_{0}$ (see \eqref{def. Sigma_epsilon_0}) we have $\rd(z,\R_+)=|\im z|$. We use the abbreviation
\begin{align}\label{def s(L,z)}
s(L,z):=\mathrm{sep}(L,\im\sqrt{z}/(d+1))
\end{align}
and set $\alpha(q):=2\max(q,q_d)$.

\begin{lemma}\label{lemma BS bounds}
Assume $q$ is in the range \eqref{range q} and $p\in[\alpha(q),\infty]$. Then the following hold.
\begin{itemize}
\item[(i)] For any $i\in [N]$,
\begin{align}\label{BSii bound}
\|BS_{ii}(z)\|_{\alpha(q)}\lesssim \omega_q(z)\|V_i\|_{L^q}.
\end{align}
\item[(ii)] For any $i,j\in [N]$, $i\neq j$, 
\begin{align}\label{BSij bound}
\|BS_{ij}(z)\|_{\alpha(q)}\lesssim \exp(-\tfrac{2}{d+1}\im\sqrt{z}\,\rd(\Omega_i,\Omega_j)) \omega_q(z)\|V_i\|_{L^q}^{1/2}\|V_j\|_{L^q}^{1/2}.
\end{align}
\item[(iii)] The diagonal part satisfies
\begin{align}\label{BSdiag bound}
\|BS_{\rm diag}(z)\|_{p}\lesssim \omega_q(z)\|V\|_{\ell^{p} L^q}.
\end{align}
\item[(iv)] The off-diagonal part satisfies
\begin{align}\label{BSoff bound}
\|BS_{\rm off}(z)\|_{\alpha(q)}\lesssim s(L,z)^2\omega_q(z)\|V\|_{\ell^{\infty} L^q},
\end{align}
\item[(v)] The full Birman--Schwinger operator satisfies
\begin{align}\label{BS bound}
\|BS_V(z)\|_{p}\lesssim \omega_q(z)\langle s(L,z)\rangle^2\|V\|_{\ell^{p} L^q}.
\end{align}
\end{itemize}
\end{lemma}

\begin{proof}
It follows from
\begin{align*}
\|BS_{\rm diag}(z)\|_{p}=\big(\sum_{i\in[N]}\|BS_{ii}(z)\|_{p}^p\big)^{1/p}
\end{align*}
that (iii) is a consequence of (i). In view of the embeddings
$\mathfrak{S}^{p_1}\subset \mathfrak{S}^{p_2}$, $\ell^{p_1}\subset \ell^{p_2}$ for $p_1\leq p_2$, (v) follows from (iii) and (iv). Moreover, (iv) follows from (ii) and the triangle inequality, using the estimate $\rd(\Omega_i,\Omega_j)\geq \frac{1}{2}L_i+\frac{1}{2}L_j$ to sum the double series. 
The estimate (i) is the same as for the $N=1$ case and follows from known results in the literature: For $q\leq q_d$, from \cite[Theorem 12]{MR3730931} for $d\geq 3$, from \cite[Theorem 4.1]{MR3608659} for $d=2$ and from \cite[Theorem 4]{AAD01} for $d=1$, for $q\geq q_d$ and $d\geq 1$ from \cite[Proposition 2.1]{MR3717979}. The bound (ii) is proved by complex interpolation as in \cite[Theorem 12]{MR3730931} in the case $q\leq q_d$ and in \cite[Proposition 2.1]{MR3717979} for $q\geq q_d$. The only difference is that we include the (second) exponential in the pointwise bound 
\begin{align}\label{pointwise bound on complex power of the resolvent}
|(-\Delta-z)^{-(a+\I t)}(x-y)|\leq C_1\e^{C_2t^2}\e^{-\im\sqrt{z}|x-y|}|x-y|^{-\frac{d+1}{2}+a}
\end{align}
for $a\in [(d-1)/2,(d+1)/2]$ and $d\geq 2$, see e.g.\ \cite[(2.5)]{MR3865141}. For $d=1$ one can use the explicit formula for the resolvent kernel. 
\end{proof}

\begin{remark}\label{remark BS bounds with Davies-Nath}
Using the results of \cite{MR4104544} (or \cite{MR1946184} in one dimension) in the proof of Lemma \ref{lemma BS bounds} we could replace the bounds \eqref{BSii bound}, \eqref{BSij bound} by the following. For $q\leq q_d$ and $i,j\in [N]$,
\begin{align}
\|BS_{ij}(z)\|_{\alpha(q)}\lesssim \exp(-\tfrac{2}{d+1}\im\sqrt{z}\,\rd(\Omega_i,\Omega_j)) |z|^{\frac{d}{2q}-1}F_{V_i,q}^{1/2}(\im\sqrt{z})F_{V_j,q}^{1/2}(\im\sqrt{z}),
\end{align}
where $F_{V,q}(s)$ is the ``Davies--Nath norm"
\begin{align}\label{def. Davies--Nath norm}
F_{V,q}(s):=\left(\sup_{y\in\R^d}\int_{\R^d}|V(x)|^{q}\exp(-s|x-y|)\rd x\right)^{\frac{1}{q}}.
\end{align}
This implies the bounds
\begin{align}
\|BS_{\rm diag}(z)\|_{p}&\lesssim |z|^{\frac{d}{2q}-1}\sup_{i\in[N]}F_{V_i,q},\\
\|BS_{\rm off}(z)\|_{\alpha(q)}&\lesssim s(L,z)^2|z|^{\frac{d}{2q}-1}\sup_{i\in[N]}F_{V_i,q},\\
\|BS_{V}(z)\|_{\alpha(q)}&\lesssim \langle s(L,z)\rangle^2|z|^{\frac{d}{2q}-1}\sup_{i\in[N]}F_{V_i,q}.
\end{align}
\end{remark}

\subsection{Norm resolvent convergence}

\begin{lemma}\label{lemma BS bound for half the potential}
Under the assumptions of Lemma \ref{lemma BS bounds} we have
\begin{align}
\||V|^{1/2}R_0(z)\|_{2p}\lesssim |\im z|^{-1} \omega_q(z)\langle s(L,z)\rangle^2\|V\|_{\ell^{p} L^q}^{1/2}.
\end{align}
\end{lemma}

\begin{proof}
The claim readily follows from \eqref{BS bound},
the identity
\begin{align*}
R_0(z)R_0(\overline{z})=\frac{1}{2\im z}(R_0(\overline{z})-R_0(z))
\end{align*}
and a $TT^*$ argument.
\end{proof}

\begin{proposition}\label{prop. nrc}
Under the assumptions of Lemma \ref{lemma BS bounds} the Schrödinger operators $H_{V^{(n)}}$ with truncated potentials converge in norm resolvent sense to $H_V$.
\end{proposition}

\begin{proof}
We first note that all the bounds of Lemma \ref{lemma BS bounds} also hold for $H_{V^{(n)}}$, uniformly in $n$. Since $V^{(n)}$ converges to $V$ in $\ell^pL^q$, it follows that the Birman-Schwinger operator associated to $H_{V^{(n)}}$ converges to $BS_V(z)$. Moreover, by Lemma \ref{lemma BS bound for half the potential}, the operators $|V^{(n)}|^{1/2}R_0(z)$ converge in $\mathfrak{S}^{2p}$-norm. Note that the square root of $V$ is trivial to compute, owing to the disjointness of supports of the $V_j$. We choose $z=\I t$, $t\gg 1$. For such $z$ the norm of the Birman-Schwinger operator is $<1$, hence a Neumann series argument and the resolvent identity \eqref{resolvent identity} yield the claim.
\end{proof}

\begin{proposition}\label{prop. nrc2}
As $s(L,z)\to 0$ for fixed $z$, the Schrödinger operators $H_{V}$ (depending on $L$) converge in norm resolvent sense to $H_{\rm diag}$.
\end{proposition}

\begin{proof}
By \eqref{BSoff bound}, it follows that $BS_{\rm off}(z)\to 0$ in norm as $s(L,z)\to 0$. The remainder of the proof is similar to that of Proposition \ref{prop. nrc}.
\end{proof}

\subsection{Proof of Klaus' theorem}

The proof is an adaptation of Klaus' original argument, which is based on the Birman-Schwinger principle. In the non-selfadjoint setting one may use e.g.\ \cite[Lemma 11.2.1]{MR2359869} as a substitute for Weyl sequences. This  yields an easy proof of the inclusion $\sigma_{\rm e}(H_V)\supset [0,\infty)\cup S$. For the converse inclusion we  prove an analogue of \cite[Proposition 2.3]{MR700696}.

\begin{proposition}\label{prop. 1 Klaus}
For $V\in \ell^{\infty}L^q$ and $z\in\C\setminus[0,\infty)$ we have 
\begin{align*}
\sigma(BS_{\rm diag}(z))=\overline{\bigcup_{i}\sigma(BS_{ii}(z))}.
\end{align*}
\end{proposition}

\begin{proof}
Only the inclusion $\subset$ is nontrivial. The resolvent set of a direct sum of bounded operators $A=\bigoplus_i A_i$ is known to be
\begin{align*}
\rho(A)=\set{\lambda\in \bigcap_i\rho(A_i)}{\sup_i\|(A_i-\lambda)^{-1}\|<\infty}.
\end{align*}
We set $A_i=BS_{ii}(z)$, whence $A=BS_{\rm diag}(z)$. Assume that $\lambda \in \C\setminus \overline{\bigcup_{i}\sigma(BS_{ii}(z))}$. This clearly implies $\lambda\in \bigcap_i\rho(A_i)$. It remains to prove that $\sup_i\|(A_i-\lambda)^{-1}\|<\infty$. By assumption, there exists $\delta>0$ such that $\rd(\lambda,\sigma(A_i))\geq \delta$. By \cite[Theorem 4.1]{MR2047381},
\begin{align*}
\|(A_i-\lambda)^{-1}\|\leq \delta^{-1}\exp(a\|A_i\|_{{\alpha}}^{{\alpha}}\delta^{-{\alpha}}+b),
\end{align*}
where $\alpha=\alpha(q)$ and $a=a(q)$, $b=b(q)$ are constants. Since, by \eqref{BSdiag bound}, $\|A_i\|_{\alpha}\lesssim \|V\|_{\ell^{\infty}L^q}$,
it follows that $\sup_i\|(A_i-\lambda)^{-1}\|<\infty$.
\end{proof}

With Proposition \ref{prop. 1 Klaus} in hand is immediate that \cite[Lemma 2.4]{MR700696} ($K_i=A_i$ in our notation) holds in the generality needed here. The Schatten bound \eqref{BSoff bound} provides a substitute for the compactness arguments \cite[Proposition 2.1-2.2]{MR700696}. For the remainder of the proof one can follow the arguments in \cite{MR700696} verbatim.

\begin{corollary}\label{cor. essential spectrum}
If $V\in \ell^pL^q$ with $q$ in the range \eqref{range q} and $p<\infty$, then we have $S=\emptyset$, i.e.\ $\sigma_{\rm e}(H_V)=[0,\infty)$. The same holds for $H_{\rm diag}$.
\end{corollary}

\begin{proof}
By \eqref{BS bound}, we have that $\lim_{t\to\infty}\|BS_V(\I t)\|=0$, which means that the inverse exists as a bounded operator for $z=\I t$ and $t\gg 1$. Corollary \ref{lemma BS bound for half the potential} implies that $R_0(z)V^{1/2}$ is compact, whence the resolvent difference \eqref{resolvent identity} is compact. The claim now follows by Weyl's theorem \cite[Theorem IX.2.4]{MR929030}.
\end{proof}

\subsection{Magnitude bounds}\label{subsection Magnitude bounds}

The following universal bounds generalize those of Theorem \ref{theorem sparse magnitude bound intro} in the introduction. They are an immediate consequence of \eqref{BS bound}, Remark \ref{remark BS bounds with Davies-Nath} and the Birman--Schwinger principle.

\begin{theorem}
Let $q$ be in the range \eqref{range q}. If $V$ is separating, then every eigenvalue $z$ of $H_V$ satisfies
\begin{align}\label{magnitude bounds 1}
\omega_q(z)^{-1} \langle s(L,z)\rangle^{-2}\lesssim \|V\|_{\ell^{\infty}L^q}
\end{align}
as well as 
\begin{align}\label{magnitude bounds 2}
|z|^{q-\frac{d}{2}}\langle s(L,z)\rangle^{-2}\lesssim \sup_{j\in[N]}F_{V_j,q}(\im\sqrt{z}).
\end{align} 
\end{theorem}

\begin{remark}
$(i)$ In the case of a single ``bump" ($N=1$) the bound \eqref{magnitude bounds 1} was proved by Frank in \cite{MR3556444} for $q\leq q_d$ and in \cite{MR3717979} for $q>q_d$. In the latter case it was observed that the inequality implies $\im z\to 0$ as $\re z\to +\infty$ for eigenvalues $z$ of $H_V$. More precisely, if we fix the norm of the potential, then
\begin{align*}
|\im z|^{1-\frac{q_d}{q}}\lesssim (\re z)^{-\frac{1}{2q}}.
\end{align*}
In the case $N\neq 1$ \eqref{magnitude bounds 1} implies that the above holds with an additional factor $\langle s(L,z)\rangle^{2}$ on the right. If $L$ grows at least polynomially, $\eta_0 L_k\gtrsim k^{\alpha}$, then we obtain
\begin{align*}
|\im z|^{1-\frac{q_d}{q}+\frac{2}{\alpha}}\lesssim (\re z)^{-\frac{1}{2q}+\frac{1}{\alpha}}\eta_0^{-\frac{2}{\alpha}},
\end{align*}
see Example a) after Definition \ref{def. sparse}. Hence, for sufficiently large $\alpha$ (depending on $q$ and $d$) the exponent of $|\im z|$ remains positive, while that of $\re z$ remains negative, and we still get the conclusion that $\im z\to 0$ as $\re z\to +\infty$.

\noindent $(ii)$ The $N=1$ case of \eqref{magnitude bounds 2} was proved in one dimension by Davies and Nath \cite{MR1946184} and in higher dimensions by the author \cite{MR4104544}. The inequality is similar to \eqref{magnitude bounds 1} for $q> q_d$. Both are relevant for ``long-range" potentials. In the case of the step potential \eqref{magnitude bounds 2} is sharp, while \eqref{magnitude bounds 1} (both for $N=1$) loses a logarithm (see \eqref{single square well norm V} and \eqref{single square well Davies Nath}). 
\end{remark}

\section{Determinant bounds}

\begin{assumption}\label{assumption 1}
Let $q$ be in the range \eqref{range q}, $p\in[2\max(q,q_d),\infty)$, $V=V(L)$ strongly separating and $\|V\|_{\ell^{p} L^q}\lesssim 1$. 
\end{assumption}

The assumption $\|V\|_{\ell^{p} L^q}\lesssim 1$ is for convenience only and can easily be removed by power counting arguments (Since the estimates are scale-invariant).

\subsection{Upper bounds}
We collect some useful estimates that we will repeatedly use (these follows from \cite[Theorem 9.2]{MR2154153}):
\begin{align}\label{Gamma}
|f(z)|\leq \exp(\mathcal{O}(1)\langle\|BS(z)\|_p\rangle^p),
\end{align}
\begin{align}\label{continuity of det}
|f_{\rm diag}(z)-f_V(z)|\leq \|BS_{\rm off}(z)\|_p\exp(\mathcal{O}(1) \langle\|BS_{\rm diag}(z)\|_p+\|BS_V(z)\|_p\rangle^p),
\end{align}
where $f=f_{\rm diag}$ or $f_V$ and $BS(z)=BS_{\rm diag}(z)$ or $BS_V(z)$. 
Formulas \eqref{Gamma}, \eqref{continuity of det}, together with the bounds of Lemma \ref{lemma BS bounds} motivates the following definitions,
\begin{align}\label{def. phi_p, psi_p}
\psi_p(t):=\exp(\mathcal{O}_p(1)\langle t\rangle^p),\quad\varphi_p(t):=t\exp(\mathcal{O}_p(1)\langle t\rangle^p),\quad t\geq 0,
\end{align}
and for $z\in \C\setminus[0,\infty)$,
\begin{equation}\label{def. Phi_p, Psi_p operator}
\begin{split}
\Psi_{p,q}(z)&:=\psi_p(\omega_q(z)\|V\|_{\ell^{p} L^q}),\\
\Phi_{p,q}(z)&:=\varphi_p(\omega_q(z)\|V\|_{\ell^{p} L^q}),\\
\Phi_{p,q}(L,z)&:=\varphi_p(s(L,z)^2\omega_q(z)\|V\|_{\ell^{p} L^q}),
\end{split}
\end{equation}
where we suppressed the dependence on $V$. We recall that $\omega_q(z)$ and $s(L,z)$ were defined in \eqref{def omegap(z)} and \eqref{def s(L,z)}. 
The constant $O_p(1)$ is allowed to vary from one occurence to another. Thus, for example, the inequality $\varphi_p(t)\psi_p(t)\lesssim \varphi_p(t)$ holds, but $\psi_p(t)\lesssim 1$ need not be true.

\subsection{Lower bounds away from zeros}

The following lemma can be considered as one of the key technical results. To state it we introduce the notation
\begin{align}
M_{p,q}(z)&:=\frac{\langle z\rangle}{|\im z|}\left(\frac{\langle z\rangle}{|z|}\right)^{5p(\frac{q_d}{q}-1)_-+8}\langle\omega_q(z)\rangle^p,\label{def. Mzeta}\\
M_{p,q}(L,z)&:=M_{p,q}(z)\langle s(L,\left(\frac{|z|}{\langle z\rangle}\right)^{5}z)\rangle^{2p},\label{def. MLzeta}
\end{align}
Moreover, we set
\begin{align}\label{def delta(z)}
\delta_H(z):=\min(\tfrac{1}{2},\rd(z,\sigma(H))).
\end{align}
In the following $H$ denotes either $H_{\rm diag}$ or $H_V$. We then write $\delta_{\rm diag}(z):=\delta_{H_{\rm diag}}(z)$ and $\delta_V(z):=\delta_{H_V}(z)$.

\begin{lemma}\label{lemma Lower bound on regularized determinants}
Suppose Assumption \ref{assumption 1} holds. Then for all $z\in \Sigma_{0}$,
\begin{align}
|f_{\rm diag}(z)|^{-1}&\leq \exp(\mathcal{O}(1) M_{p,q}(z)\log\frac{1}{\delta_{\rm diag}(z)}),\label{eq. Lower bound on regularized determinants fdiag}\\
|f_{V}(z)|^{-1}&\leq \exp(\mathcal{O}(1) M_{p,q}(L,z)\log\frac{1}{\delta_V(z)}).\label{eq. Lower bound on regularized determinants fV}
\end{align}
\end{lemma}

\begin{proof}
In the following, $f$ denotes either $f_{\rm diag}$ or $f_V$. We are going to apply Lemma~\ref{lemma wedge} with parameters (in the notation of that lemma) 
\begin{align}
&r_1^2=c| z|,\quad r_2^2=c^2| z|,\quad r_3^2=c^3| z|,\\
&R_1^2=c^{-1}| z|,\quad R_2^2=c^{-2}\langle z\rangle,\\
&2\varphi_1=\epsilon|\arg( z)|,\quad 2\varphi_2=\epsilon^2|\arg( z)|,\quad 2\varphi_3=c\epsilon^2|\arg( z)|,\\
&\theta_1=\pi-\varphi_1,\quad \theta_2=\pi-\varphi_2,\quad \theta_3=\pi-\varphi_3,
\end{align}
and with $U_j$, $j=1,2,3$, the wedges defined in \eqref{def. wedge}. Here $c\ll 1$ is a small absolute constant and $\epsilon=\epsilon(z)$ is a small parameter that will be chosen momentarily.
Note that $|\arg( z)|\ll 1$ since $ z\in\Sigma_{0}$ and thus $\sin(2\varphi_j)\asymp \varphi_j\asymp \tan(2\varphi_j)$, which will be used repeatedly in the proof. We will now verify the assumptions of Lemma~\ref{lemma wedge}. The first condition in \eqref{condition lemma wedge} is satisfied by definition. Using that, for $j=1,2$, 
\begin{align*}
\rd(\partial U_j,\partial U_3)
=\sin(2\varphi_j-2\varphi_3)r_j^2\asymp (c\epsilon)^{j}|\im z|,
\end{align*}
we find that the second and third condition in \eqref{condition lemma wedge} become
\begin{align*}
c\epsilon\ll (c^2(|z|/\langle z\rangle)^{1/2})^{\frac{2\pi}{\theta_3-\varphi_3}+2},\quad c^3\epsilon|\im z|/\langle z\rangle (c^2(|z|/\langle z\rangle)^{1/2})^{\frac{2\pi}{\theta_3-\varphi_3}}\ll 1,
\end{align*}
respectively, which means that the third condition is trivially satisfied. Since $c^2(|z|/\langle z\rangle)^{1/2}\ll 1$ and $\frac{2\pi}{\theta_3-\varphi_3}\leq 3$, the second condition is satisfied if we choose e.g.\ $\epsilon= c^{10}(|z|/\langle z\rangle)^{5/2}$, which we do.
We will next show that, for a suitable choice of $z_2\in U_2$,
\begin{equation}\label{max less than omegazeta}
\max_{w\in U_3}\log|f(w)|-\log|f( z_2)|\lesssim_c (\langle s(L,\epsilon^2 z)\rangle^{2\nu}\epsilon^{-2(\frac{q_d}{q}-1)_-}\langle\omega_q( z)\rangle)^p,
\end{equation}
where $\nu=0$ if $f=f_{\rm diag}$ and $\nu=1$ if $f=f_{V}$.
Since in the present case
\begin{align*}
\frac{R_2^2}{\rd(\partial U_2,\partial U_3)} \left(\frac{R_2}{r_2}\right)^{\frac{2\pi}{\theta_3-\varphi_3}}\lesssim c^{-28}\frac{\langle z\rangle}{|\im z|}\left(\frac{\langle z\rangle}{| z|}\right)^{8},
\end{align*}
we see that \eqref{eq. lemma wedge} holds.
Lemma~\ref{lemma wedge} thus implies that \eqref{eq. Lower bound on regularized determinants fdiag}, \eqref{eq. Lower bound on regularized determinants fV} hold. To prove~\eqref{max less than omegazeta} we first observe that, by the maximum principle, $|f|$ attains its maximum on the boundary of the wedge $U_3$. We estimate this on the boundary component corresponding to the ray $w\rho^2\e^{2\I\varphi_3}$, $\rho>r_3$, the estimates for the other two boundary components being similar. By \eqref{Gamma} and the bounds of Lemma \ref{lemma BS bounds}, this maximum is bounded by the right hand side of \eqref{max less than omegazeta}, where we used that $\sup_{\rho>r_3}\omega_q(\rho^2\e^{2\I\varphi_3})\lesssim_c \epsilon^{-2(\frac{q_d}{q}-1)_-}\omega_q(z)$ and that $L$ is strongly separating. This proves \eqref{max less than omegazeta} for the first term on the left. The other part follows by selecting, for instance, $z_2=\tfrac{\I}{4}R_2^2$ and using the estimates (similar to \eqref{continuity of det}, see \cite[Theorem 9.2]{MR2154153})
\begin{align*}
|f_{\rm diag}(z_2)-1|\lesssim \Phi_{p,q}(z_2),\quad |f_{V}(z_2)-1|\lesssim \Phi_{p,q}(L,z_2),
\end{align*}
where we once again used Lemma \ref{lemma BS bounds} and $\omega_q(z_2)\lesssim \omega_q(z)$. Note that in the case of $f_V$ we can absorb the factor $\langle s(L,z_2)\rangle^2$ in the definition of $\Phi_{p,q}(L,z_2)$ into the $O(1)$ term since $\im\sqrt{z_2}\gtrsim 1$ and $L$ is strongly separating. Since $\omega_q(z_2)\ll 1$ for $R_2\gg 1$ (which is true whenever $c\ll 1$), it follows that $|f(z_2)|\geq 1/2$ for $c$ sufficiently small. This finishes the proof of \eqref{max less than omegazeta}.
\end{proof}

\subsection{Upper bound on the resolvent away from eigenvalues}
As a consequence of Lemma \ref{lemma Lower bound on regularized determinants} we also obtain an upper bound for the resolvent of $H_V$ away from the spectrum. The idea is to use the following infinite-dimensional analogue of Caramer's rule (see \cite[(7.10)]{MR482328}),
\begin{align}\label{Cramers rule}
\|(I+BS(z))^{-1}\|\leq \frac{\exp(\mathcal{O}(1)\|BS(z)\|_{p}^p)}{|f(z)|}.
\end{align}

\begin{proposition}\label{prop. upper bound for the resolvent away from the spectrum}
Suppose Assumption \ref{assumption 1} holds. Then for all $z\in \Sigma_{0}$,
\begin{align}
\|(H_{\rm diag}-z)^{-1}\|\leq \exp(\mathcal{O}(1) M_{p,q}(z)\log\frac{1}{\delta_{\rm diag}(z)}),\label{upper bound for the resolvent Hdiag}\\
\|(H_V-z)^{-1}\|\leq \exp(\mathcal{O}(1) M_{p,q}(L,z)\log\frac{1}{\delta_V(z)}),\label{upper bound for the resolvent HV}
\end{align}
where $M_{q}(z)$, $M_{q}(L,z)$ are given by \eqref{def. Mzeta}.
\end{proposition}

\begin{proof}
In view of the trivial bound 
\begin{align*}
\|(-\Delta-z)^{-1}\|\leq |\im z|^{-1},
\end{align*}
the claim follows from Lemma \ref{lemma upper bound for the BS away from the spectrum}, the resolvent identity \eqref{resolvent identity}, Corollary \ref{lemma BS bound for half the potential} and the fact that $|\im z|^{-1}\lesssim M_{p,q}(z)$.
\end{proof}

\begin{lemma}\label{lemma upper bound for the BS away from the spectrum}
The operator norms of $(I+BS_{\rm diag}(z))^{-1}$ and $(I+BS_{V}(z))^{-1}$ are bounded by the right hand side of \eqref{upper bound for the resolvent Hdiag} and \eqref{upper bound for the resolvent HV}, respectively.
\end{lemma}

\begin{proof}
We only prove the claim for $BS_{\rm diag}(z)$; the other part is similar. By \eqref{Cramers rule}, \eqref{BS bound} and \eqref{def. Phi_p, Psi_p operator}, we have $\|(I+BS_{\rm diag}(z))^{-1}\|\lesssim|f_{\rm diag}(z)|^{-1}\Psi_{p,q}(z)$. Since $\Psi_{p,q}(z)\lesssim M_{p,q}(z)$, Lemma \ref{lemma Lower bound on regularized determinants} implies that the latter is bounded by \eqref{upper bound for the resolvent Hdiag}.
\end{proof}

\section{Comparison between $H_{\rm diag}$ and $H_V$}\label{Section Comparison between Hdiag and HV}
\subsection{Ghershgorin type upper bounds}

We record the following Ghershgorin type bound. 
We temporarily restore the norm of the potential and define
\begin{align}
\omega_{q,i}(z):=\omega_q(z)\|V_i\|_{L^q},
\end{align}
and $M_{p,q,i}(z)$ is defined by \eqref{def. Mzeta} with $\omega_{q}(z)$ replaced by $\omega_{q,i}(z)$.

\begin{proposition}
Under Assumption \ref{assumption 1} the discrete spectrum of $\sigma(H_{V})$ in $\Sigma_0$ is contained in the set
\begin{align}\label{Ghershgorin set}
\bigcup_{i\in[N]}\set{z\in\C}{\im\sqrt{z}L_i-\log \langle s(L,z)\rangle\lesssim - M_{p,q,i}(z)\log\delta_{H_{V_i}}(z)+\log \omega_{q,i}(z)}.
\end{align}
\end{proposition}

\begin{proof}
Assume first that $N<\infty$, and consider the Hilbert spaces $\mathcal{H}_n=L^2(\Omega_n)$, $\mathcal{H}=\bigoplus_{n\in[N]}\mathcal{H}_n$, with operators $A_{ij}=\delta_{ij}I_{\mathcal{H}}+BS_{ij}(z)$ and $A=(A_{ij})_{i,j=1}^{N}$. Applying the Gershgorin theorem for bounded block operator matrices due to Salas \cite{MR1685637} (see also \cite[Theorem 1.13.1]{MR2463978}) yields
\begin{align*}
\sigma(A)\subset\bigcup_{i=1}^N\set{\lambda\in\C}{\|(A_{ii}-\lambda)^{-1}\|^{-1}\leq \sum_{j\in[N]\setminus\{i\}}\|A_{ij}\|}.
\end{align*}
Note that here we are using the convention that $\|(A_{ii}-\lambda)^{-1}\|=\infty$ if $\lambda\in \sigma(A_{ii})$. 
By the Birman--Schwinger principle, this implies that
\begin{align*}
\sigma(H_V)\subset\bigcup_{i=1}^N\set{z\in\C}{\|(I_{\mathcal{H}}+BS_{ii}(z))^{-1}\|^{-1}\leq \sum_{j\in[N]\setminus\{i\}}\|BS_{ij}(z)\|}.
\end{align*}
Again, we include the spectrum of $A_{ii}$ in the set on the right.
By Lemma \ref{lemma upper bound for the BS away from the spectrum}, we have 
\begin{align*}
\|(I_{\mathcal{H}}+BS_{ii}(z))^{-1}\|\leq \exp(\mathcal{O}(1) M_{p,q,i}(z)\log\frac{1}{\delta_{H_{V_i}}(z)}).
\end{align*}
On the other hand, by \eqref{BSij bound} and the strong separation property,
\begin{align*}
\sum_{j\in[N]\setminus\{i\}}\|BS_{ij}(z)\|\lesssim  s(L,z)\exp(-\tfrac{1}{d+1}\im\sqrt{z}L_i)\omega_q(z)
\end{align*}
The claim for $N<\infty$ follows. Similarly, it follows for the truncated operators $H^{(n)}$. Since the set \eqref{Ghershgorin set} is independent of $n$ the claim for the case $N=\infty$ then follows from the norm resolvent convergence of the truncated operators (Proposition \ref{prop. nrc}).
\end{proof}

\subsection{Lower bounds: Qualitative results}
In the following we establish criteria on the sequence $(L_j)_j$ that guarantee proximity of $\sigma_{\rm d} (H_{V})$ to $\sigma_{\rm d} (H_{\rm diag})$ in various regions of the spectral plane. 
Let us fist discuss some standard facts from perturbation theory (\cite[Chaper IV]{Ka}). It is well knows that the spectrum of a closed operator $H$ is upper semicontinuous under perturbations, and the same is true for each separated part of the spectrum \cite[Theorem 3.16]{Ka}. Moreover, a \emph{finite system} of eigenvalues $\{\zeta_1,\ldots,\zeta_n\}$ changes continuously, just as in the finite-dimensional case. This follows from the fact that the Riesz projection
\begin{align}\label{Riesz projection}
\frac{1}{2\pi\I}\int_{\Gamma}(\zeta-H)^{-1}\rd\zeta
\end{align}
is continuous in $H$ in the uniform topology (in the sense of generalized convergence of operators). Here $\Gamma$ is a closed contour (a piecewise smooth curve) in the resolvent set of $H$ and encircling the eigenvalues $\zeta_1,\ldots,\zeta_n$ (and no other point of the spectrum) once in the counterclockwise direction. Then 
\begin{align}\label{trace Riesz proj.}
\frac{1}{2\pi\I}\Tr \int_{\Gamma}(\zeta-H)^{-1}\rd\zeta=n.
\end{align}
Here we consider a finite system of eigenvalues of $H_{\rm diag}$ in some compact subset $\Sigma\Subset\C\setminus[0,\infty)$. By Corollary \ref{cor. essential spectrum}, Assumption \ref{assumption 1} implies that each point in $\Sigma$ is either in the resolvent set or a discrete eigenvalue of $H_{\rm diag}$. By compactness, $\Sigma\cap \sigma(H_{\rm diag})$ is a finite set. 
We then have
\begin{align}\label{def delta0(U)}
\delta_0(\Sigma):=\min\set{\rd(\zeta,\sigma(H_{\rm diag})\setminus\{\zeta\})}{\zeta \in \Sigma\cap \sigma(H_{\rm diag})}>0.
\end{align}
For $\delta\in (0,\tfrac{1}{3}\min(1,\delta_0(\Sigma)))$ we set
\begin{align}\label{def Udelta Gammadelta}
U_{\delta}:=\bigcup_{\zeta \in \Sigma\cap \sigma(H_{\rm diag})}D(\zeta,\delta),\quad \Gamma_{\delta}:=\partial U_{\delta}.
\end{align}
In general it is hard to determine $\delta_0(\Sigma)$, but we still have $\Gamma_{\delta}\subset\rho(H_{\rm diag})$ for generic~$\delta$. This is all that is needed for a lower bound on the number of eigenvalues in $U_{\delta}$.
The norm resolvent convergence (Proposition \ref{prop. nrc2}) implies the following proposition. In Subsection \ref{subsection Proof of Prop. comparison principle first version} we will give an alternative proof using the argument principle.

Let us first state our assumptions for the remainder of this section.

\begin{assumption}\label{assumption 2}
Let $\Sigma \subset\Sigma_0\cap\C\setminus[0,\infty)$ be a compact subset, let $\delta_0(\Sigma),U_{\delta},\Gamma_{\delta}$ be defined by \eqref{def delta0(U)}, \eqref{def Udelta Gammadelta} and let $\delta\in(0,\delta_0(\Sigma))$.
\end{assumption}

\begin{proposition}\label{Prop. comparison principle first version}
Suppose Assumptions \ref{assumption 1}, \ref{assumption 2} hold. Then for any $\delta\in (0,\delta_0(\Sigma))$ there exists a constant $C=C(\delta,\Sigma)$ such that, if $s(L,\zeta)\leq 1/C$, then 
\begin{align*}
N(H_V;U_{\delta})=N(H_{\rm diag};U_{\delta}).
\end{align*}
\end{proposition}

\subsection{Argument principle}\label{subsection argument principle}
The argument in the previous subsection involved compactness and continuity and is obviously non-quantitative. The issue is of course the need for a quantitative estimate of the resolvent on the curve $\Gamma_{\delta}$. We will prove such estimates in Proposition~\ref{prop. upper bound for the resolvent away from the spectrum}. Here we argue in a slightly different (albeit closely related) manner. We will use the regularized Fredholm determinants (see for instance \cite[IV.2]{MR0246142}, \cite[Chapter 9]{MR2154153} or \cite[XI.9.21]{MR1009163})
\begin{align}
f_{\rm diag}(z)&:=\Det_{p}(I+BS_{\rm diag}(z)),\label{f}\\
f_V(z)&:=\Det_{p}(I+BS_V(z)),\label{fdiag}
\end{align}
where $p\in [2\max(q,q_d),\infty)$ is assumed to be an integer. The main property that we will use is that the $f_{\rm diag},f_V$ are analytic functions	in $\C\setminus[0,\infty)$ and have zeros (counted with multiplicity) exactly at the eigenvalues of $H_{\rm diag}$, $H_V$, respectively. Moreover, by the generalized argument principle (see e.g.\ \cite[Theorem 3.2]{MR3717979} or \cite[Theorem 6.7]{behrndt2020generalized}),
\begin{align}\label{argument principle}
N(H;U_{\delta})=\frac{1}{2\pi\I}\int_{\Gamma_{\delta}}\frac{\rd}{\rd z}\log f(\zeta)\rd \zeta,
\end{align}
where $H=H_{\rm diag}$ or $H_V$ and $f=f_{\rm diag}$ or $f_V$.
This suggests a comparison between $H_{\rm diag}$ and $H_{V}$ via Rouch\'e's theorem (see e.g.\ \cite{MR1001877} for related ideas). 
We set 
\begin{align}\label{def r}
r_{\delta}:=\sup_{z\in \Gamma_{\delta}}\frac{|f_{\rm diag}(z)-f_V(z)|}{|f_{\rm diag}(z)|}.
\end{align}
We will show that $r_{\delta}<1$ if $\max_{j\in[n]}s(L,\zeta_j)$ is sufficiently small. Rouch\'e's theorem then asserts that $f_{\rm diag}$ and $f$ have the same number of zeros in $U_{\delta}$.

\subsection{Alternative proof of Proposition \ref{Prop. comparison principle first version}}\label{subsection Proof of Prop. comparison principle first version}

Without loss of generality we may assume that $\Sigma$ contains exactly one eigenvalue $\zeta$ of $H_{\rm diag}$. 
We are going to use Lemma \ref{lemma appendix simplest version}. For this purpose we set $U_1=U_{\delta}$ and let $U_2\subset\C\setminus[0,\infty)$ be a precompact simply connected neighborhood of $U_1$ containing a point $\zeta_2\notin \sigma(H_{\rm diag})$. This is possible by \eqref{sparse magnitude bound} applied to $V_j$ (i.e.\ with $N=1$) since we can take $\zeta_2=-A$, where $A\gg 1$. By \eqref{continuity of det} we find that
\begin{align}\label{fdiag-f}
\sup_{z\in \Gamma_{\delta}}|f_{\rm diag}(z)-f_V(z)|
\leq C_1\,s(L,\zeta)^2,
\end{align}
where $C_1=C_1(\delta,\Sigma)$. We take $A$ so large that
\begin{align}\label{Upsilon}
\Phi_{p,q}(z_0)\leq\tfrac{1}{2}.
\end{align}
This is possible since $\lim_{A\to\infty}\omega_q(-A)=0$ by \eqref{BSdiag bound}. By Lemma \ref{lemma appendix simplest version} there exists a constant $C_2=C_2(\delta,\Sigma)$ such that
\begin{align}\label{lower bound fdiag simplest}
\log|f_{\rm diag}(z)|\geq  -C_2\quad \mbox{for all}\quad z\in \Gamma_{\delta}.
\end{align}
Here we used that $\log|f_{\rm diag}(z_0)|\geq -\log 2$, which follows from \eqref{Upsilon}, \eqref{BSdiag bound} and \eqref{Gamma}.
Combining \eqref{fdiag-f} and \eqref{lower bound fdiag simplest}, we infer that $r_{\delta}<1$ if $s(L,\zeta)$ is sufficiently small.

\subsection{Lower bounds: Quantitative results}
In the following we establish quantitative versions of Proposition \ref{Prop. comparison principle first version}. 

We return to estimating the quantity $r_{\delta}$ in~\eqref{def r} featuring in Rouch\'e's theorem. 

\begin{lemma}
Suppose Assumptions \ref{assumption 1}, \ref{assumption 2} hold. Then
\begin{align}\label{upper bound formula for r improved}
r_{\delta}\leq \max_{\zeta\in\Sigma} s(L,\zeta)^2\exp\big(\mathcal{O}(1) \langle s(L,\zeta)\rangle^2 M_{p,q}(\zeta)\log\frac{1}{\delta}\big),
\end{align}
where $\delta_{H_{\rm diag}}(\zeta)$, $s(L,\zeta)$, $M_{p,q}(\zeta)$ are given by \eqref{def delta(z)}, \eqref{def s(L,z)}, \eqref{def. Mzeta}, respectively.
\end{lemma}

\begin{proof}
Again we may assume that $\Sigma$ contains exactly one eigenvalue $\zeta$ of $H_{\rm diag}$, so that $U_{\delta}=D(\zeta,\delta)$ and $\Gamma_{\delta}=\partial D(\zeta,\delta)$. 
We clearly have $\delta(z)=\delta$ and $\omega_q(z)\asymp \omega_q(\zeta)$ for $z\in\Gamma_{\delta}$.
It is easy to see that \eqref{continuity of det} and Lemma \ref{lemma BS bounds} imply
\begin{align}
\sup_{z\in\Gamma_{\delta}}|f_{\rm diag}(z)-f_V(z)|\lesssim \Phi_{p,q}(L,\zeta)\Psi_{p,q}(\zeta).
\end{align}
We have also used that $L$ is strongly separating, hence $s(L,z)\lesssim s(L,\zeta)$. 
In order to estimate $r_{\delta}$ it remains to bound  $|f_{\rm diag}(z)|$ from below using \eqref{eq. Lower bound on regularized determinants fdiag}.
\end{proof}

As an immediate corollary we obtain an improvement of Proposition
\ref{Prop. comparison principle first version}.

\begin{proposition}\label{Prop. comparison principle improvement}
Suppose Assumptions \ref{assumption 1}, \ref{assumption 2} hold. Then
\begin{align*}
N(H_V;U_{\delta})=N(H_{\rm diag};U_{\delta}),
\end{align*}
provided that $L$ is so large that $r_{\delta}<1$ in \eqref{upper bound formula for r improved}. 
\end{proposition}

\section{From quasimodes to eigenvalues}\label{section From quasimodes to eigenvalues}

\subsection{Existence of a single eigenvalue}
We record a useful corollary of Proposition \ref{prop. upper bound for the resolvent away from the spectrum}. The proof is obvious.

\begin{corollary}\label{corollary from quasimodes to eigenvalues}
Suppose Assumption \ref{assumption 1} holds.
Assume that there is a normalized $\psi\in L^2(\R^d)$ such that $$\|(H_V-\zeta)\psi\|\leq \epsilon,$$ where $\epsilon^{-1}$ is larger than the right hand side of \eqref{upper bound for the resolvent HV} at $z=\zeta$.
Then $\sigma_{\rm d}(H_V)\cap D(\zeta,\delta))\neq\emptyset$.
\end{corollary}

\subsection{Existence of a sequence of eigenvalues}

If instead of a single quasi-eigenvalue we consider a sequence $(\zeta_j)_j$ with $\lim_{j\to\infty}\im\sqrt{\zeta_j}=0$, an across-the-board assumption like  the one in Corollary \ref{corollary from quasimodes to eigenvalues} is not feasible since the right hand side of \eqref{upper bound for the resolvent HV} at $z=\zeta_n$ tends to infinity as $n\to\infty$. One possible solution would be to modify the previous arguments and select $L$ as a function of the sequence $(\zeta_j)_j$. We will follow a similar, albeit slightly different strategy which we find more intuitive. It is also closer in spirit to the inductive argument in \cite{MR3627408}, which is based on strong resolvent convergence. Once more, the approach we will outline can be viewed as a quantitative version of that method. 

The strategy will be to first construct quasimodes of $H_V$ in a direct way (Lemma \ref{lemma quasimodes Hdiag to HV}) and then use Corollary \ref{corollary from quasimodes to eigenvalues} to obtain existence of eigenvalues. The quasi-eigenvalues and quasimodes will be actual eigenvalues and eigenfunctions of $H_{\rm diag}$. We introduce a sequence of scales $\eps_j$ and $a_j$, where $\eps_j$ has dimension of energy and $a_j$ has dimension of length, and assume that the eigenfunctions $\psi_j$ corresponding to $\zeta_j$ decay exponentially away from $\Omega_j$ in such a way that
\begin{align}\label{assumption exponential decay}
\|V_i\psi_j\|\leq C_q a_j^{-d/\widetilde{q}}\exp(-c_0\,\im\sqrt{\zeta_j}\,\rd\,(\Omega_i,\Omega_j))\|V_i\|_{L^{\widetilde{q}}},
\end{align}
where $\widetilde{q}\geq 2$ and $c_0>0$. In the following applications we can take $c_0=1$. We will then choose $L$ such that
\begin{align}\label{choice of Lj}
\im\sqrt{\zeta_n} L_n\geq C\log\left(n\log^2\langle n\rangle\sup_{j\in[n]}\eps_j^{-1}a_j^{-d/\widetilde{q}}\sup_{i\in[n]}\|V_i\|_{L^{\widetilde{q}}}\right),
\end{align} 
where $C=C(d,\widetilde{q})$ is a large constant.

\begin{lemma}\label{lemma quasimodes Hdiag to HV}
Assume that $V\in \ell^{\infty}L^{\widetilde{q}}$  for some $\widetilde{q}\geq 2$ and that
$\zeta_j$ are eigenvalues of $H_{V_j}$ with normalized eigenfunctions $\psi_j$ satisfying \eqref{assumption exponential decay}. Then there exists an absolute constant $C=C(d,q)$ such that if $V(L)$ is separating and satisfies \eqref{choice of Lj}, then $H_V$ has a sequence of normalized quasimodes $\psi_j$,
\begin{align*}
\|(H_{V}-\zeta_j)\psi_j\|\leq \eps_j.
\end{align*}
\end{lemma}

\begin{remark}
Lemma \ref{lemma quasimodes Hdiag to HV} could be seen as a quantitative version of Lemma 2 in \cite{MR3627408}.
\end{remark}

\begin{proof}
Let $\psi_j$ be the eigenfunctions of $H_{V_j}$ corresponding to $\zeta_j$, i.e.\ $(H_{V_j}-\zeta_j)\psi_j=0$. For $n\in[N]$ we make the following (stronger) induction hypothesis $P(n)$:
\begin{align}\label{P(n)}
\|(H^{(n)}-\zeta_j)\psi_j\|\leq \eps_j\big(1-\frac{1}{\log\langle n\rangle}\big),\quad j\in [n]
\end{align}
(recall \eqref{def. Vuppern} for the definition of $H^{(n)}$).
The base case $n=1$ is true by assumption. 
Assume now that $P(n-1)$ holds. By the exponential decay \eqref{assumption exponential decay},
\begin{align*}
\|(H^{(n)}-\zeta_{n})\psi_{n}\|&\leq \sum_{j=1}^{n-1}\|V_j\psi_{n}\|\leq nC_q a_n^{-d/\widetilde{q}}\exp(-c_0\eta_n\,L_n)\sup_{j\in[n]}\|V_j\|_{L^{\widetilde{q}}},
\end{align*}
where we have set $\eta_n:=\im\sqrt{\zeta_n}$.
Moreover, by induction hypothesis, for $j\in [n-1]$,
\begin{align*}
\|(H^{(n)}-\zeta_{j})\psi_{j}\|&\leq \|(H^{(n-1)}-\zeta_{j})\psi_{j}\|+\|V_{n}\psi_{j}\|\\
&\leq \eps_j\big(1-\frac{1}{\log\langle n-1\rangle}\big)+C_q a_j^{-d/\widetilde{q}}\exp(-c_0\eta_j\,L_n)\|V_n\|_{L^{\widetilde{q}}}.
\end{align*}
Hence $P(n)$ would hold if $L_n$ satisfied the estimates
\begin{align}
&nC_q a_n^{-d/\widetilde{q}}\exp(-c_0\eta_n\,L_n)\sup_{j\in[n]}\|V_j\|_{L^{\widetilde{q}}}\leq \eps_n,\label{accomplished our goal 1}\\
&C_q a_j^{-d/\widetilde{q}}\exp(-c_0\eta_j\,L_n)\|V_n\|_{L^{\widetilde{q}}}\leq \eps_j\big(\frac{1}{\log\langle n-1\rangle}-\frac{1}{\log\langle n\rangle}\big),\label{accomplished our goal 2}
\end{align}
for $j\in [n]$.
By the mean value theorem 
\begin{align*}
\big(\frac{1}{\log\langle n-1\rangle}-\frac{1}{\log\langle n\rangle}\big)\gtrsim \frac{1}{n\log^2\langle n\rangle},
\end{align*}
and it is easy to check that \eqref{accomplished our goal 1}--\eqref{accomplished our goal 2} are satisfied for the choice \eqref{choice of Lj}. This completes the induction step. For $N<\infty$ the claim now follows from \eqref{P(n)} with $n=N$. Now consider the case $N=\infty$. 
Since $L$ is separating,
\begin{align*}
\lim_{n\to\infty}\|(H_V-H^{(n)})\psi_j\|^2&=\lim_{n\to\infty}\|(V-V^{(n)})\psi_j\|^2=\lim_{n\to\infty}\sum_{k=n+1}^{\infty}\|V_k\psi_j\|^2\\
&\leq a_j^{-d/\widetilde{q}}\|V\|_{\ell^{\infty}L^{\widetilde{q}}}^2\lim_{n\to\infty}\sum_{k=n+1}^{\infty}\exp(-2\eta_j L_k)=0.
\end{align*}
Together with \eqref{P(n)} this yields the claim for $N=\infty$.
\end{proof}

\begin{remark}\label{remark nlog2n}
The factor $n\log^2 n$ in \eqref{choice of Lj} comes from the induction hypothesis and should not be taken too seriously. One could of course replace $\log\langle n\rangle$ by any other slowly varying sequence tending to infinity. However, this would not change the bound \eqref{choice of Lj} significantly. 
\end{remark}

\subsection{Quasimode construction}

We now construct the potential $W_j$ having $\zeta_j$ as an eigenvalue. 

\begin{lemma}\label{lemma single square well}
Given $\zeta\in\Sigma_{0}$ and $x_0\in\R^d$ there exists a potential $W=W(\zeta,x_0)\in L^{\infty}_{\rm comp}(\R^d)$ such that the following hold.
\begin{enumerate}
\item $H_W$ has eigenvalue $\zeta$;
\item $\supp W\subset B(x_0,R)$, where 
\begin{align}\label{single square well R}
R=R(\zeta)\asymp |\zeta|^{\frac{1}{2}}|\im\zeta|^{-1}|\log|\im\zeta|/|\zeta||.
\end{align}
\item For any $1\leq q\leq \infty$
\begin{align}\label{single square well norm V}
\|W\|_{L^q(\R^d)}\asymp |\zeta|^{\frac{d}{2q}}|\im\zeta|^{1-\frac{d}{q}}|\log^{\frac{d}{q}}|\im\zeta/\zeta||.
\end{align} 
\item For $1\leq q\leq q_d$,
\begin{align}\label{single square well Davies Nath}
F_{W,q}(\im\sqrt{\zeta})\lesssim |\zeta|^{\frac{d}{2q}}|\im\zeta|^{1-\frac{d}{q}}
\end{align}
\item The normalized eigenfunction $\psi=\psi(\zeta,x_0)$ of $H_W$ corresponding to $\zeta$ satisfies the exponential decay estimate
\begin{align}\label{exponential decay single square well}
|\psi(x)|\leq C|\zeta|^{1/4} |x|^{-\frac{d-1}{2}}\exp(-\im\sqrt{\zeta}\,\rd\,(x,\supp W)).
\end{align}
\end{enumerate}
\end{lemma} 

\begin{proof}
By scaling it suffices to prove this for $|\zeta|\asymp 1$.
In view of the results of Section \ref{section Complex square wells} (and $\zeta=E$ in the notation of that section) we can then simply choose $W$ as a shifted step potential. The shift of course does not affect the eigenvalues nor the $L^q$ norms. The latter are trivial to compute using the size bound $|W|=\mathcal{O}(\epsilon)$ and the formula \eqref{single square well R} for the width of the step. The estimate \eqref{single square well Davies Nath} follows from a direct computation. The exponential decay follows from Lemma \ref{lemma exponential decay} or the explicit form of the wavefunction for the step potential.
\end{proof}

\begin{remark}
Similar results involving complex step potentials are contained in \cite{MR3627408}, \cite{MR4104544}, \cite{MR4157680}, albeit in a less quantitative form. A technical detail that distinguishes our proof from these is that we first pick the eigenvalue, then find the potential. This avoids the use of Rouch\'e's theorem in  \cite{MR4104544}, \cite{MR4157680}.
\end{remark}

\subsection{Exponential decay}

We prove that the exponential decay bound \eqref{assumption exponential decay} holds for a class of compactly supported potentials that will be relevant in the next section. The important point here is that the constant $C$ in \eqref{exponential decay} is independent of $W$.

\begin{lemma}\label{lemma exponential decay}
Assume that $\supp W\subset B(0,R)$ and $\zeta\in [0,\infty)$, $\im\sqrt{\zeta}\leq \tfrac{1}{2}R^{-1}\log R$, $|\zeta|^{1/2}\geq K R^{-1}$ for a large absolute constant $K$. Assume that $\psi$ is a normalized eigenfunction of $H_W$ with eigenvalue $\zeta$. Then there exists an absolute constant $C=C(d)$ such that for $|x|>R$,
\begin{align}\label{exponential decay}
|\psi(x)|\leq C|\zeta|^{1/4} |x|^{-\frac{d-1}{2}}\exp(-\im\sqrt{\zeta}|x|).
\end{align}
\end{lemma}

\begin{proof}
Since $\psi$ is normalized in $L^2$ it has units $l^{-d/2}$. By homogeneity, we may thus assume that $|\zeta|=1$. Since $\psi$ solves the Helmholtz equation
\begin{align*}
-\Delta\psi(x)=\kappa^2\psi(x)
\end{align*}
for $|x|>R$ and $\kappa^2=\zeta$, we have (see e.g.\ \cite[Chapter 1, Section 2]{MR2598115})
\begin{align*}
\psi(x)=A|x|^{-\nu}H^{(1)}_{\nu}(\kappa|x|)
\end{align*}
in this region, where $H^{(1)}_{\nu}$ is the Hankel function, $\nu=(d-2)/2$ and $A=A(d,W)$ is a normalization constant. By the well-known asymptotics of the Hankel function at infinity,
\begin{align*}
\psi(x)=Ac_d|x|^{-\frac{d-1}{2}}\exp(-\im\kappa|x|)(1+\mathcal{O}(|x|^{-1})).	
\end{align*}
This would imply \eqref{exponential decay} if we could show that $A$ has an upper bound independent of $W$. Since $\psi$ is normalized,
\begin{align*}
A^2c_d^2\int_{|x|>R}|x|^{-(d-1)}\exp(-2\im\kappa|x|)(1+\mathcal{O}(|x|^{-1}))\rd x\leq \|\psi\|^2=1.
\end{align*}
For sufficiently large $K$ we estimate the integral from below by
\begin{align*}
(1-\mathcal{O}(K^{-1}))\int_R^{2R}\exp(-2\im\kappa r)\rd r\geq (1-\mathcal{O}(K^{-1}))R\exp(-\log R)\geq \frac{1}{4},
\end{align*}
which proves that $A\leq 2/c_d$.
\end{proof}

\begin{corollary}\label{cor exp decay}
Assume that $V_j(x)=W_j(x-x_j)$ and that the assumptions of Lemma \ref{lemma exponential decay} are satisfied for $W_j,\zeta,\psi_j,R_j$. Then \eqref{assumption exponential decay} holds for any $\widetilde{q}\geq 2$ and with
\begin{align}\label{def. aj}
a_j^{-d/\widetilde{q}}\lesssim|\zeta_j|^{1/4}\Big(\frac{L_j}{\im\sqrt{\zeta_j}}\Big)^{-\frac{d-1}{2}(\frac{1}{2}-\frac{1}{\widetilde{q}})}.
\end{align}
\end{corollary}

\begin{proof}
Let $\tfrac{1}{\widetilde{q}}+\tfrac{1}{r}=\tfrac{1}{2}$. By Hölder,
\begin{align*}
\|V_i\psi_j\|\leq \|V_i\|_{L^{\widetilde{q}}}\|\psi_j\|_{L^r(\Omega_i)}
\end{align*}
and by \eqref{exponential decay}, 
\begin{align*}
\|\psi_j\|_{L^r(\Omega_i)}\lesssim |\zeta_j|^{1/4}\big(\int_{\Omega_i}(|x-x_j|^{-\frac{d-1}{2}}\exp(-\im\sqrt{\zeta_j}\,|x-x_j|))^r\rd x\big)^{1/r}.
\end{align*}
Since $|x-x_j|\geq \rd(x,\Omega_j)$, 
\begin{align*}
\|\psi_j\|_{L^r(\Omega_i)}\lesssim |\zeta_j|^{1/4}\Big(\frac{L_j}{\im\sqrt{\zeta_j}}\Big)^{-\frac{d-1}{2r}}.
\end{align*}
The claim follows.
\end{proof}

\subsection{A quantitative version of Boegli's example}\label{subsection A quantitative version of Boegli's result}

In view of Corollary \ref{corollary from quasimodes to eigenvalues}, given $\zeta_j\in \sigma_{H_{V_j}}$ and $\delta_j>0$ we would like to choose $\eps_j=\epsilon_j(\zeta_j,\delta_j,L)$ as
\begin{align}\label{def. epsilon_n}
\eps_j^{-1}=\exp(\mathcal{O}(1) M_{p,q}(L,\zeta_j)\log\frac{1}{\delta_j})
\end{align}
and require that \eqref{choice of Lj} holds with $a_j$ as in \eqref{def. aj}. This gives a sufficient condition on the sequence $L$ ensuring that $\rd(\zeta_j,\sigma(H_V))\leq \delta_j$. The following proposition follows immediately from Corollary \ref{corollary from quasimodes to eigenvalues}, Lemma \ref{lemma quasimodes Hdiag to HV} and Corollary \ref{cor exp decay}. 

\begin{proposition}\label{prop. inverse result}
Suppose Assumption \ref{assumption 1} holds, $V_j\in \ell^{\infty}L^{\widetilde{q}}$ for some $\widetilde{q}\geq 2$ and that $\supp V_j(\cdot+x_j)\subset B(0,R_j)$ for some positive $R_j$. Let $\zeta_j,\delta_j$ be sequences satisfying $\im\sqrt{\zeta_j}\leq \tfrac{1}{2}R_j^{-1}\log R_j$, $|\zeta_j|^{1/2}\geq K R_j^{-1}$ for some large absolute constant $K$ and $\delta_j\in(0,1/2)$. Assume that $L$ satisfies \eqref{choice of Lj} with $a_j$ as in \eqref{def. aj} and $\epsilon_j$ as in \eqref{def. epsilon_n}.
If $\zeta_j\in \Sigma_{0}$ is an eigenvalue of $H_{V_j}$ of multiplicity $m_j$, then $D(\zeta_j,\delta_j)$ contains at least $m_j$ eigenvalues of $H_V$, counted with multiplicity.
\end{proposition}

In the following we apply Proposition \ref{prop. inverse result} with $V_j=W(\zeta_j,x_j)$, where $W$ is the complex step potential in Lemma \ref{lemma quasimodes Hdiag to HV}. Clearly, $V_j\in L^q(\R^d)$ for every $q\in[1,\infty]$, with
\begin{align}
\|V\|_{\ell^pL^q}&\lesssim \big(\sum_n\big(|\zeta_n|^{\frac{d}{2q}}|\im\zeta_n|^{1-\frac{d}{q}}|\log^{\frac{d}{q}}|\im\zeta_n/\zeta_n||\big)^p\big)^{\frac{1}{p}},\label{VellpLq}\\
\sup_{j\in [n]}\|V_j\|_{L^{\infty}}&\lesssim \sup_{j\in [n]}|\im\zeta_n|.\label{Vjellinfty}
\end{align}
We will also take $\widetilde{q}=\infty$, so that $a_j^{-d/\widetilde{q}}=1$. For the remainder of this section we assume the following.

\begin{assumption}\label{assumption 3}
Let $q>d$, and assume that \eqref{condition on sequence zetan intro} holds. Without loss of generality we may and will also assume that $\im\zeta_n$ is monotonically decreasing.
\end{assumption}

\begin{lemma}\label{lemma condition implies}
Under Assumption \eqref{assumption 3} the following hold.
\begin{enumerate}
\item[(i)] $\|V\|_{\ell^pL^q}\leq\epsilon_1$, $\|V\|_{\ell^{\infty}L^{\infty}}\lesssim 1$.
\item[(ii)] $|\im\zeta_n|\lesssim 1$.
\item[(iii)] $\langle\zeta_n\rangle\lesssim |\im\zeta_n|^{-\frac{2(q-d)}{d}}$.
\item[(iv)] $M_{p,q}(\zeta_n)\gtrsim |\im\zeta_n|^{p(\frac{(q-d)}{dq}+\frac{q_d}{q}-1)}$.
\item[(v)] $M_q(\zeta_n)\lesssim |\im\zeta_n|^{-1-\frac{2(q-d)}{d}-5p(\frac{q_d}{d}-1)-8+p(\frac{(q-d)}{dq}+\frac{q_d}{q}-1)}$.
\end{enumerate}
\end{lemma}

\begin{proof}
Condition \eqref{condition on sequence zetan intro} states that the right hand side of \eqref{VellpLq} with $p=q$ is bounded by $\epsilon_1$. Since $p>q$ and the embedding $\ell^{q}\subset\ell^{p}$ is contractive, the first claim in (i) follows. Since $|\im\zeta_n|\leq |\zeta_n|$ and $|\log|\im\zeta_n/\zeta||\geq 1$ for $\zeta_n \in\Sigma_0$, Condition \eqref{condition on sequence zetan intro} also implies
\begin{align}\label{firstboud secondbound}
|\im\zeta_n|^{1-\frac{d}{2q}}\leq\epsilon_1,\quad |\zeta_n|^{\frac{d}{2}}\leq \epsilon_1|\im\zeta_n|^{q-d}.
\end{align}
Since $q>d$ the first bound implies (ii) and thus the second claim in (i) follows from \eqref{Vjellinfty}. The claim (iii) follows from the second bound in \eqref{firstboud secondbound} and (ii). Using~(iii), we find
\begin{align}
\omega_q(\zeta_n)=|\zeta_n|^{-\frac{1}{2q}}|\im\zeta_n|^{\frac{q_d}{q}-1}\gtrsim |\im\zeta_n|^{\frac{(q-d)}{dq}+\frac{q_d}{q}-1}.\label{estimate omegaqzetan} 
\end{align}
This yields (iv) since $M_{p,q}(\zeta_n)\geq \omega_q(\zeta_n)^p$. It also follows from (ii) that
\begin{align}\label{2 also yields}
\frac{|\zeta_n|}{\langle \zeta_n\rangle}\gtrsim |\im\zeta_n|.
\end{align}
Combining (iii), \eqref{estimate omegaqzetan} and \eqref{2 also yields} with the trivial lower bound $|\zeta_n|\geq |\im\zeta_n|$ in~\eqref{def. Mzeta} yields~(v).
\end{proof}

\begin{remark}
From the first equality in \eqref{estimate omegaqzetan} and the definition of $M_{p,q}(z)$ in \eqref{def. Mzeta} it is easy to see that for $|\zeta_n|\asymp 1$, we have better bounds
\begin{align}\label{better bounds for Mpq if zeta order 1}
|\im\zeta_n|^{p(\frac{q_d}{q}-1)}\lesssim M_{p,q}(\zeta_n)\lesssim |\im\zeta_n|^{p(\frac{q_d}{q}-1)-1}.
\end{align}
\end{remark}

\begin{lemma}\label{lemma Lk kalpha}
Suppose that $L_k\gtrsim k^{\alpha}$. Then, under Assumption \ref{assumption 3},
\begin{align}
s(L,\left(\frac{|\zeta_n|}{\langle \zeta_n\rangle}\right)^{5}\zeta_n)\rangle\lesssim |\im\zeta_n|^{-\frac{1}{\alpha}(\frac{7}{2}+\frac{(q-d)}{d})},\quad
M_{p,q}(L,\zeta_n)\lesssim |\im\zeta_n|^{-\kappa_{\alpha}},
\end{align}
where $\kappa_{\alpha}:=1+\frac{2(q-d)}{d}+5p(\frac{q_d}{d}-1)+8-p(\frac{(q-d)}{dq}+\frac{q_d}{q}-1)+\frac{2p}{\alpha}(\frac{7}{2}+\frac{(q-d)}{d})$.
\end{lemma}

\begin{proof}
Combining \eqref{2 also yields} with the estimate
\begin{align}\label{estimate imsqrzetan}
\im\sqrt{\zeta_n}\asymp\frac{|\im\zeta_n|}{|\zeta_n|^{1/2}}\gtrsim |\im\zeta_n|^{1+\frac{(q-d)}{d}},
\end{align}
where the first bound holds since $\zeta_n\in\Sigma_0$ and the second bound follows from Lemma \ref{lemma condition implies} (iii), we obtain 
\begin{align}
s(L,\left(\frac{|\zeta_n|}{\langle \zeta_n\rangle}\right)^{5}\zeta_n)\rangle\lesssim \mathrm{sep}(L,|\im\zeta_n|^{\frac{7}{2}+\frac{(q-d)}{d}}).
\end{align}
The claim thus follows from Proposition \ref{lemma distribution function} and Example a) following it.
\end{proof}

\begin{remark}
For $|\zeta_n|\asymp 1$, we again have better bounds
\begin{align}\label{better bounds sLzeta if zeta order 1}
s(L,\left(\frac{|\zeta_n|}{\langle \zeta_n\rangle}\right)^{5}\zeta_n)\rangle&\lesssim |\im\zeta_n|^{-\frac{1}{\alpha}},\\
M_{p,q}(L,\zeta_n)&\lesssim |\im\zeta_n|^{p(\frac{q_d}{q}-1)-1-\frac{2p}{\alpha}}.
\end{align}
\end{remark}
We assume now that
\begin{align}\label{assumption deltan}
\delta_n\geq \exp(-|\im\zeta_n|^{-\gamma})
\end{align}
for some $\gamma>0$. This lower bound is motivated from the corresponding upper bound that results from the Ghershgorin estimate \eqref{Ghershgorin set} and a posteriori by \eqref{choice of Lj 3}.

\begin{lemma}
Fix a compact set $\Sigma\subset\Sigma_0\cap \C\setminus[0,\infty)$. The there exists $c=c(\Sigma)$ such that
\begin{align}
\sigma(H_V)\cap\Sigma\subset\set{z}{\delta_{H_{V_n}}(z)\leq \exp(-c L_n)}.
\end{align}
\end{lemma}

The following lemma is obvious.

\begin{lemma}\label{lemma logepsilonn}
If $\epsilon_n$ is defined by \eqref{def. epsilon_n}, $L_k\gtrsim k^{\alpha}$ and $\delta_n$ satisfies \eqref{assumption deltan}, then under Assumption \ref{assumption 3},
\begin{align*}
\log\epsilon_n^{-1}\lesssim |\im\zeta_n|^{-\kappa_{\alpha}-\gamma}.
\end{align*}
\end{lemma}

Lemma \eqref{lemma condition implies} (i) and Lemma \ref{lemma logepsilonn} imply that the right hand side of \eqref{choice of Lj} (with $\widetilde{q}=\infty$) is bounded by $|\im\zeta_n|^{-\kappa_{\alpha}-\gamma}\log\langle n\rangle$. We will show that 
\begin{align}\label{n<imzetanegpower}
\langle n\rangle\leq |\im\zeta_n|^{-\frac{d}{2}-q+1},
\end{align}
for all but finitely many $n\in\N$, which will then give a sufficient condition for the choice of $L$ in Proposition \ref{prop. inverse result}, namely
\begin{align}\label{choice of Lj 2}
L_n\geq C|\im\zeta_n|^{-\kappa_{\alpha}-\gamma-1-\frac{(q-d)}{d}}\log(|\im\zeta_n|^{-1}).
\end{align}
Here we have used \eqref{estimate imsqrzetan} to estimate $\im\sqrt{\zeta_n}$ from below. In order to be consistent with our assumption $L_k\gtrsim k^{\alpha}$ we actually choose
\begin{align}\label{choice of Lj 3}
L_n= C|\im\zeta_n|^{-\widetilde{\kappa}},
\end{align}
where, in view of \eqref{n<imzetanegpower}, it suffices to take
\begin{align}\label{tildekappa}
\widetilde{\kappa}:=\max(\kappa_{\alpha}+\gamma+2+\frac{(q-d)}{d},\alpha(\frac{d}{2}+q-1)).
\end{align}
The exact choice of $\alpha$ is not important for us and we choose $\alpha=1$ for convenience. 

\begin{lemma}
Under Assumption \ref{assumption 3} we have \eqref{n<imzetanegpower} for all but finitely many $n\in\N$.
\end{lemma}

\begin{proof}
Suppose the claim is false. Then there exists a subsequence, again denoted by $(\zeta_n)_n$, such that $\langle n\rangle>|\im\zeta_n|^{-\frac{d}{2}-q+1}$. Since $|\zeta_n|\geq |\im\zeta_n|$, \eqref{condition on sequence zetan intro} implies
\begin{align*}
\sum_n\langle n\rangle^{-1}<\sum_n|\im\zeta_n|^{\frac{d}{2}+q-1}\lesssim 1,
\end{align*}
a contradiction.
\end{proof}

\subsection{Proof of Theorem \ref{thm quant Boegli intro}}

We now specialize Proposition \ref{prop. inverse result} to the step potential $V_j=W(\zeta_j,x_j)$ and the explicit choice \eqref{choice of Lj 3}, which will prove Theorem \ref{thm quant Boegli intro}. Since we already know that the exponential decay bound is true for these potentials (see \eqref{exponential decay single square well}) we do not need to check the conditions $\im\sqrt{\zeta_j}\leq \tfrac{1}{2}R_j^{-1}\log R_j$, $|\zeta_j|^{1/2}\geq K R_j^{-1}$, but it is easy to see from \eqref{single square well R} that they do hold.

\begin{proposition}\label{prop. inverse result step potential}
Suppose Assumption \ref{assumption 3} holds, $\delta_n>0$ satisfies \eqref{assumption deltan} for some $\gamma>0$, and let $V=V(L)$ be the potential whose bumps $V_n=W(\zeta_n,x_n)$ are separated by $L_n$ in \eqref{choice of Lj 3}. Then $D(\zeta_n,\delta_n)$ contains an eigenvalues of $H_V$. Moreover, $\|V\|_{L^q}\leq \epsilon_1$ and $V$ decays polynomially,
\begin{align}
|V(x)|\lesssim \langle x\rangle^{-\frac{1}{\widetilde{\kappa}}},
\end{align}
where $\widetilde{\kappa}$ is given by \eqref{tildekappa} for some arbitrary $\alpha>0$.
\end{proposition}

\begin{proof}
By \eqref{single square well norm V}, we have the bound $|V(x)|\lesssim |\im\zeta_n|$ for $|x-x_n|\leq R(\zeta_n)$ and zero elsewhere. Since $|\im\zeta_n|\lesssim 1$ by Lemma \ref{lemma condition implies} (i), it follows that $V$ is bounded. Since $\widetilde{\kappa}\geq 2$, a comparison between $L_n$ and $|\Omega_n|=R(\zeta_n)$ in \eqref{single square well R} shows that $V$ is sparse. Therefore, by \eqref{lower bound Li in terms of xi, Omegai sparse}, we have $L_n\gtrsim |x_n|$. Hence, \eqref{choice of Lj 3} yields
\begin{align*}
|V(x_n)|\lesssim |\im\zeta_n|\lesssim L_n^{-\frac{1}{\widetilde{\kappa}}}\lesssim x_n^{-\frac{1}{\widetilde{\kappa}}},
\end{align*}
from which the decay bound follows.
\end{proof}

\section{Complex step potential}\label{section Complex square wells}

In this section we will establish precise estimates for eigenvalues of the sperically symmetric complex step potential $V=V_0\mathbf{1}_{B(0,R)}$, where $V_0\in \C$ and $R>0$. The bound state problem for $V_0<0$ and $d=1,3$ is treated in virtually any quantum mechanics textbook (see e.g.\ Problem 25 and Problem 63 in \cite{MR1746199}). 
We adopt the notation
\begin{align}\label{step potential notation}
\chi=\sqrt{E},\quad \kappa=\sqrt{\chi^2-V_0}. 
\end{align}
Here $E\in \C$ is the eigenvalue parameter, i.e.\ we consider the stationary Schrödinger equation
\begin{align}\label{Schrödinger equation square well}
-\Delta \psi+(V-E)\psi=0,
\end{align}
which becomes $-\Delta \psi-\kappa^2\psi=0$
inside the step and 
$-\Delta \psi-\chi^2\psi=0$
outside the step. 

\subsection{One dimension}\label{subsection one dimension step potential}
We start with one-dimensional case. The solution space to~\eqref{Schrödinger equation square well} then splits into even and odd functions, while in higher dimensions it splits into functions with definite angular momentum $\ell$. We consider odd functions as these also provide a solution for the case $d=3$ and $\ell=0$ ($s$-waves). The standard procedure to solving the square well problem reduces the task to finding zeros of the nonlinear scalar function $F(V_0,\kappa):=\I\chi-\kappa\cot(\kappa R)$, where $\chi=\sqrt{\kappa^2+V_0}$ by~\eqref{step potential notation}. A complete study of all the complex poles of this equation was initiated by Nussenzveig \cite{Nussenzveig} for $V_0\in \R\setminus\{0\}$. Subsequent articles in the physics literature \cite{JOFFILY1973301}, \cite{PhysRevC.53.2004}, \cite{Dabrowski_1997}, \cite{PhysRevA.61.032716} investigated the case of complex potentials. The solution $\kappa=\kappa(V_0)$ is not single-valued as there are branch points where $\partial F/\partial\chi=0$. The viewpoint endorsed by \cite{PhysRevA.61.032716} is to regard the equation $F(V_0,\kappa)=0$ as the definition of a Riemann surface. This approach treats the complex variables $\kappa$ and $V_0$ on equal footing. In fact, it is easy to see that one can always use $\kappa$ as a coordinate, i.e.\ one can solve for $V_0$,
\begin{align}\label{solve for V0 1d}
V_0=-\kappa^2\sec^2(\kappa R).
\end{align}
For the purpose of the construction of the sparse potential in Subsection \ref{subsection A quantitative version of Boegli's result} we do not need to solve for $\kappa$. Instead, we pick $\kappa$ first and then define $V_0$ by \eqref{solve for V0 1d}. To get an eigenvalue (i.e.\ a resonance on the physical sheet) we simply need to take care of the condition $\im\chi>0$, i.e. 
\begin{align}\label{take care of im chi}
\re (\kappa R\cot(\kappa R))>0.
\end{align}
We are only interested in complex eigenvalues $E$ with $|E|\asymp 1$ (the general case can be obtained by scaling). We will try to make $V_0$ in \eqref{solve for V0 1d} small, i.e.\ we postulate that $V_0=\epsilon \widetilde{V}_0$, where $\epsilon>0$ is a small parameter and $\widetilde{V}_0\in \C$ is of unit size. By~\eqref{step potential notation} this implies that $|\kappa|\asymp 1$, and \eqref{solve for V0 1d} then reveals that $|\sin^2(\kappa R)|\asymp \epsilon^{-1}$, which means that
\begin{align}\label{dominant balance exponentials}
\e^{2\im\kappa R}\asymp \epsilon^{-1},\quad \e^{-2\im\kappa R}\asymp \epsilon.
\end{align}
Since we are free to choose $\kappa$, 
set $\kappa=\pm 1+\I\epsilon\sigma$ with $\sigma> 0$, which will yield an eigenvalue with $\re E=1+\mathcal{O}(\epsilon)$ and $|\im E|=\mathcal{O}(\epsilon)$.
Going back to \eqref{dominant balance exponentials} we see that we must have
\begin{align}\label{R consistent with Davies bound}
R=\frac{1}{2\sigma\epsilon}\log\frac{C}{\epsilon}
\end{align}
for some constant $C$. It is quickly checked that this is consistent with the bound of Abramov et al.\ \cite{AAD01} since $\|V\|_{L^1}\gtrsim \log \frac{1}{\epsilon}$ and $|E|\asymp 1$. In view of \eqref{dominant balance exponentials} we may write $\e^{2\I\kappa R}=\epsilon u$, where $u=C\e^{2\I\re \kappa R}$. Using the Taylor approximation
\begin{align}\label{sec2 Taylor}
\sec^2(\kappa R)=-4\epsilon u(1+2\epsilon u +\mathcal{O}(\epsilon^2))
\end{align}
we obtain from \eqref{solve for V0 1d} that
\begin{align}\label{V0 result 1d}
V_0=4\epsilon u(\re\kappa)^2+\mathcal{O}(\epsilon^2),
\end{align}
which provides the desired smallness $|V_0|=\mathcal{O}(\epsilon)$. As already mentioned, we need to make sure that \eqref{take care of im chi} holds. Using the Taylor approximation
\begin{align*}
\cot(\kappa R)=\I(1+2\epsilon u+\mathcal{O}(\epsilon^2)),
\end{align*}
we find that \eqref{take care of im chi} holds if
\begin{align}\label{condition u}
-(\re\kappa)(\im u)-(\im\kappa)(\re u)>0
\end{align}
In particular, for $u\in \I\R_+$, we find that \eqref{condition u} forces us to choose $\re\kappa=-1$. Adopting this choice for $u$, it is then easy to check that we get an eigenvalue with $\re E=1+\mathcal{O}(\epsilon)$ and $\im E=\epsilon(4|u|-2\sigma)+\mathcal{O}(\epsilon^2)$ as desired. By simple scaling arguments this proves the one-dimensional case of Lemma \ref{lemma single square well}. We observe that the result is consistent with the trivial numerical range bound $\im E\leq \im V_0$; in fact, by choosing $\sigma$ small, $E$ can be taken arbitrarily close to the boundary of the numerical range $\im z=4\epsilon|u|$, up to errors of order $\epsilon^2$. 

 Before we conclude the one-dimensional case we note that the same result could have been obtained with an even wavefunction, in which case $\sec$ replaced by $\csc$ in \eqref{solve for V0 1d} and $\cot$ is replaced by $\tan$ in \eqref{take care of im chi}. The Taylor approximations
\begin{align}
\csc^2(\kappa R)=4\epsilon u(1+\mathcal{O}(\epsilon)),\quad
\tan(\kappa R)=\I(1-2\epsilon u+\mathcal{O}(\epsilon^2)),
\end{align}
and the freedom to choose the signs and the imaginary part of $u$ yields a proof of Lemma \ref{lemma single square well} using odd solutions.

\subsection{Higher dimensions}
By symmetry reductions we are led to consider the radial Schrödinger equation
\begin{align*}
(-\partial_r^2-\frac{d-1}{r}\partial_r+\frac{\ell(\ell+d-2)}{r^2}+V(r)-E)\psi_{\ell}(r)=0.
\end{align*}
It can be shown (see e.g.\ \cite[(5.12)]{MR2648080}) that an eigenvalue $E$ corresponds to a zero of the function (Wronskian)
\begin{align}\label{Wronskian}
F(V_0,\kappa):=\kappa J_{\nu}'(\kappa R)H^{(1)}_{\nu}(\chi R)-\chi J_{\nu}(\kappa R)H^{(1)'}_{\nu}(\chi R),
\end{align}
where $\nu=\ell+\frac{d-2}{2}$. We recall that $\chi,\kappa,E,V_0$ are related by \eqref{step potential notation}. Computations of resonances for spherically symmetric potentials can be 
found in \cite{MR115692}, \cite{MR987299}, \cite{MR1953522}, \cite{MR2648080}. The last three papers use uniform asymptotic expansion of Bessel functions for large order. Here we only consider $s$-waves, i.e.\ $\ell=0$. Then we have the asymptotics 
\begin{align}
J_{\nu}(z)&=\big(\frac{2}{\pi z}\big)^{1/2}\cos(z-\frac{\pi\nu}{2}-\frac{\pi}{4})(1+\mathcal{O}(|z|^{-1})),\\
H^{(1)}_{\nu}(z)&=\big(\frac{2}{\pi z}\big)^{1/2}\exp(\I z-\I\frac{\pi\nu}{2}-\I\frac{\pi}{4})(1+\mathcal{O}(|z|^{-1})),\\
J_{\nu}'(z)&=-\big(\frac{2}{\pi z}\big)^{1/2}\sin(z-\frac{\pi\nu}{2}-\frac{\pi}{4})(1+\mathcal{O}(|z|^{-1})),\\
H^{(1)'}_{\nu}(z)&=\I\big(\frac{2}{\pi z}\big)^{1/2}\exp(\I z-\I\frac{\pi\nu}{2}-\I\frac{\pi}{4})(1+\mathcal{O}(|z|^{-1}).
\end{align}
With the same choice of $\kappa$ as in the one-dimensional case and with $u\in\C$ such that $\e^{2\I(\kappa R-\frac{\pi\nu}{2}-\frac{\pi}{4})}=\epsilon u$, we then obtain that the zeros of $F(V_0,\kappa)$ coincide with the zeros of a function 
\begin{align}
\kappa\sin(\omega(\kappa R))-\I\chi\cos(\omega(\kappa R))+\mathcal{O}(R^{-1}),
\end{align} 
where $\omega(\kappa R)=\kappa R-\frac{\pi\nu}{2}-\frac{\pi}{4}$ and the $\kappa$-derivative of the error term is $\mathcal{O}(1)$. Recall that $\chi=\chi(V_0,\kappa)$ is given by \eqref{step potential notation}. The zeros of the function without the error term are found exactly as in the one-dimensional case and can be parametrized by $\kappa$, v.i.z.
\begin{align}\label{V_0 result higher d}
V_0=-\kappa^2\csc^2(\omega(\kappa R)).
\end{align}
This follows by dividing the above expression by $\cos(\omega(\kappa R))$ which has no zeros since $\im\kappa>0$. Since $|\cos(\omega(\kappa R))^{-1}|=\mathcal{O}(\epsilon^{1/2})$ we get from \eqref{R consistent with Davies bound} that the error after dividing is $\mathcal{O}(\epsilon^{3/2})$, i.e.\ we are looking for the zeros of a function
\begin{align}\label{Ftilde}
\widetilde{F}(V_0,\kappa)=V_0+\kappa^2\csc^2(\omega(\kappa R))+\mathcal{O}(\epsilon^{3/2}),
\end{align} 
where the derivative of the error is $\mathcal{O}(\epsilon^{1/2})$.
The implicit function theorem thus yields $\partial \widetilde{F}(V_0,\kappa)/\partial V_0=1+\mathcal{O}(\epsilon^{1/2})$, which means that we can solve $\widetilde{F}(V_0,\kappa)$ for $V_0$, and the solution satisfies \eqref{V_0 result higher d} up to errors $\mathcal{O}(\epsilon^{3/2})$. Hence we obtain that $|V_0|=\mathcal{O}(\epsilon)$ as before.

\subsection{Proof of Theorem \ref{thm. violation of locality into}} We return to one dimension. We first prove the upper bound \eqref{Vsp}. Since $\sqrt{|V_0|}R_0$ is of order one, the bound in \cite{MR3556444} yields that the total number of eigenvalues of $H_{V_j}$ is also of order one. By Proposition \ref{Prop. comparison principle first version} (it is clear that the assumption on the norm of $V$ can be dropped), given $N\gg 1$, we can find $L=L(N)$ such that $H_{\rm s}$ has the same number of eigenvalues in $\Sigma=\Sigma(N)$ as $H_{\rm diag}$, which is just the $N$-fold orthogonal sum of the $H_{V_j}$ and hence has less than $\mathcal{O}(N)$ eigenvalues by the first part of the argument.

To prove the lower bound \eqref{Vnsp} we return to the formulas \eqref{solve for V0 1d}, \eqref{take care of im chi}, but we now fix $V_0=\I$. We also set $R=NR_0$ and $R_0\asymp 1$, so that the dimensionless parameter $\sqrt{|V_0|}R$ is of size $N$. We first solve an approximate equation and then use Rouch\'e's theorem to show that the exact equation \eqref{solve for V0 1d} has solutions close to the approximate ones. Finally, we use \eqref{take care of im chi} to check that we have found a pole on the physical plane (i.e.\ an eigenvalue). The approximate equation is $G_1(\kappa)=0$, where
\begin{align*}
G_1(\kappa):=V_0-4\kappa^2\e^{2\I\kappa R},
\end{align*}
and the approximation will be valid in the regime $\im\kappa R\gg 1$. Since $G_1$ can be factored,
\begin{align*}
G_1(\kappa)=(\sqrt{V_0}-2\kappa\e^{\I\kappa R})(\sqrt{V_0}+2\kappa\e^{\I\kappa R}),
\end{align*}
we only look for zeros of the first factor.
These zeros $\kappa_n$ are expressed by means of the Lambert $W$ function,
\begin{align*}
\kappa_n R=-\I W_n(\I \sqrt{V_0}R/2),
\end{align*}
where $n\in\Z$ and $W_n$ are the branches of the Lambert $W$ function. According to in \cite[(4.19)]{MR1414285} the asymptotic expansion of $W_n(z)$ as $|z|\to\infty$ is
\begin{align}\label{asymptotics Lambert W}
W_n(z)=\log(z)+2\pi\I n-\log(\log(z)+2\pi\I n)+\mathcal{O}(\frac{\log(\log(z)+2\pi\I n)}{\log(z)+2\pi\I n}),
\end{align}
where $\log$ is the principal branch of the logarithm on the slit plane with the negative real axis as branch cut. For $z=\I \sqrt{V_0}R/2$ this gives
\begin{align}\label{kappan formula1}
\kappa_nR=2\pi n-\I\log(\I \sqrt{V_0}R/2)+\I\log(\log(\I \sqrt{V_0}R/2)+2\pi\I n)+E_n(V_0,R),
\end{align}
where, for $N\gg 1$, the error satisfies $|E_n(V_0,R)|\lesssim \log(\log N+|n|)/(\log N+|n|)$, where we recalled that $\sqrt{|V_0|}R\asymp N$. For the assumption $\im\kappa R\gg 1$ made before to be consistent with the formula for $\kappa_n$ we require 
\begin{align*}
\re\log\left(\frac{\log(\I \sqrt{V_0}R/2)+2\pi\I n}{\I \sqrt{V_0}R/2}\right)\gg 1\iff \left|\frac{\log(\I \sqrt{V_0}R/2)+2\pi\I n}{\I \sqrt{V_0}R/2}\right|\gg 1.
\end{align*}
Since $N\gg 1$ we can neglect the logarithm in the second expression and deduce the condition $|n|\gg N$, which we will assume henceforth. This gives us the error bound $|E_n(V_0,R)|\lesssim \log |n|/|n|$, which implies that
\begin{align}\label{imkappa=logn}
\im\kappa_n R\gtrsim \log \frac{|n|}{N},
\end{align}
in agreement with the assumption $\im\kappa R\gg 1$. We also obtain the more precise formulas
\begin{align}\label{kappan}
\re \kappa_nR=2\pi n+\frac{5\pi}{4}+\mathcal{O}(\log |n|/|n|),\quad
\im\kappa_n R=\log \frac{|n|}{N}+\mathcal{O}(1), 
\end{align}
where we assumed that $n<0$. To justify this assumption, we recall from the discussion at the end of Subsection \ref{subsection one dimension step potential} that \eqref{V0 result 1d} and \eqref{condition u}, together with \eqref{kappan} and the assumption $V_0=\I$ made at the beginning of this subsection, imply that $n$ must be negative. 


Having found the large zeros of $G_1(\kappa)$ we proceed to find those of 
\begin{align*}
G_2(\kappa):=V_0+\kappa^2\sec^2(\kappa R),
\end{align*}
which determines the eigenvalues of the step potential (see the beginning of Subsection \ref{subsection one dimension step potential}). We define $\epsilon_n:=\exp(-2\im\kappa_n R)$, so that $\e^{2\I\kappa_n R}=\epsilon_n u_n$ for some $u_n$ on the unit circle. Note that, by \eqref{kappan}, $\epsilon_n=\mathcal{O}(1)\big(\frac{N}{|n|}\big)^2$. Using~\eqref{sec2 Taylor} with $\epsilon=\epsilon_n$, $u=u_n$, we estimate, for $\widetilde{\epsilon}_n\ll 1$, 
\begin{align}\label{G1-G2}
\sup_{|\kappa-\kappa_n|=\widetilde{\epsilon}_n}|G_2(\kappa)-G_1(\kappa)|=\mathcal{O}(n^2\epsilon_n^2).
\end{align}
Moreover, for $|\kappa-\kappa_n|=\widetilde{\epsilon}_n$ we have
\begin{align*}
|G_1'(\kappa)|&\gtrsim |\kappa|^2R\e^{-2\im\kappa R}\gtrsim Nn^{2}\epsilon_n,\\
|G_1''(\kappa)|&\lesssim |\kappa|^2R^2\e^{-2\im\kappa R}\lesssim N^2n^{2}\epsilon_n.
\end{align*}
Using $G_1(\kappa_n)=0$ and Taylor expanding, it follows that
\begin{align}
|G_1(\kappa)|\gtrsim Nn^{2}\epsilon_n\widetilde{\epsilon}_n+\mathcal{O}(N^2n^{2}\epsilon_n\widetilde{\epsilon}_n^2).
\end{align}
For this to be meaningful we must of course assume $\widetilde{\epsilon}_n\ll 1/N$, which we do. Then we have $|G_1(\kappa)|\gtrsim Nn^{2}\epsilon_n\widetilde{\epsilon}_n$. Comparing this with \eqref{G1-G2} we see that 
\begin{align*}
\sup_{|\kappa-\kappa_n|=\widetilde{\epsilon}_n}|G_1(\kappa)|^{-1}|G_2(\kappa)-G_1(\kappa)|<1,
\end{align*}
provided $\widetilde{\epsilon}_n\gg \epsilon_n/N$. Adopting the choice $\widetilde{\epsilon}_n=C\frac{N}{n^2}$, where $C$ is a large constant, we see that there exists a zero $\widetilde{\kappa}_n\in D(\kappa_n,C\frac{N}{n^2})$ of $G_2$. By the smallness of $N/n^2$, it follows that $\widetilde{\kappa}_n$ also satisfies \eqref{kappan}. We drop the tilde, i.e.\ we now denote the zeros of $G_2$ by $\kappa_n$. Summarizing what we have done so far, we have found infinitely many resonances $\kappa_n$, $|n|\gg N$, of of the step potential satisfying \eqref{kappan}. The last step is to check which of the resonances lie on the physical sheet, i.e.\ are actual eigenvalues. For this we need to check condition \eqref{take care of im chi}. By \eqref{kappan},
\begin{align*}
&\cot(\kappa_n R)=-\I\big(1+\I\big(\frac{N}{|n|}\big)^2\e^{\mathcal{O}(1)+\I\mathcal{O}(\log|n|/|n|)}+\mathcal{O}(\big(\frac{N}{|n|}\big)^4)\big),\\
&\re (\kappa_n R\cot(\kappa_n R))=2\pi\frac{N^2}{n}\e^{\mathcal{O}(1)}(1+\mathcal{O}(\log^2|n|/|n|))+\log\frac{|n|}{N}(1+\mathcal{O}(\log|n|/|n|)).
\end{align*}
Hence, recalling that $n<0$, the condition \eqref{take care of im chi} is nonvoid and is satisfied whenever $|n|\log\frac{|n|}{N}\ll N^2$; this holds for $|n|\ll N^2/\log N$. Recalling \eqref{step potential notation} we obtain the complex energies $E=E_n$,
\begin{align*}
\re E_n\asymp \frac{n^2}{N^2},\quad \im E_n\asymp \frac{|n|}{N^2}\log\frac{|n|}{N}, 
\end{align*}
and those energies with $c\frac{N^2}{\log N}\leq |n|\leq C\frac{N^2}{\log N}$ lie in the rectangle $\Sigma$ (see Theorem~\ref{thm. violation of locality into}).
This completes the proof of the lower bound \eqref{Vnsp}.

\section{Technical tools}\label{section Technical tools}
\subsection{Lower bounds on moduli of holomorphic functions}\label{subsection holomorphic functions}

We collect some well known results about the modulus of holomorphic functions away from zeros, based on Cartan's bound for polynomials (see e.g.\ \cite{MR589888}).

Let $U_1\Subset U_2\Subset \C$, where $U_2$ is simply connected. Assume that $f$ is holomorphic in a neighborhood of $U_2$ and $\zeta_2\in U_2$. Let $z_1, z_2\ldots, z_n$, be the zeros of $f$ in $U_2$. Define
\begin{align*}
Z_{f,\delta,U_2}:=\bigcup_{j=1}^n D(z_j,\delta).
\end{align*}

The following version can be found in \cite[Appendix D]{MR3969938}. 

\begin{lemma}\label{lemma appendix simplest version}
There exists a constant $C=C(U_1,U_2,z_0)$ such that for any sufficiently small $\delta>0$,
\begin{align*}
\log|f(z)|\geq -C\log\frac{1}{\delta}\big(\log|f(z)|-\log\max_{z\in U_2}|f(\zeta_2)|\big)
\end{align*}
for all $z\in U_1\setminus Z_{f,\delta,U_2}$.
\end{lemma}

We need also use a more precise version, where $f$ is holomorphic in a neighborhood of $U_3$,
\begin{align}\label{Uj=D(0,rj)}
U_j=D(0,r_j),\quad j=1,2,3,
\end{align}
with $r_1<r_2<r_3$ and $\zeta_2=0$. 

\begin{lemma}\label{lemma Levin}
Assume \eqref{Uj=D(0,rj)}. Then there exists an absolute constant $C$ such that for any sufficiently small $\delta>0$,
\begin{align*}
\log|f(z)|\geq -C\log\frac{1}{\delta}\big((r_2-r_1)^{-1}+\log^{-1}\big(\frac{r_3}{r_2}\big)\big)\big(\log\max_{|z|=r_3}|f(z)|-\log|f(\zeta_2)|\big)
\end{align*}
for all $z\in U_1\setminus Z_{f,\delta r_2,U_2}$.
\end{lemma}

\begin{proof}
The proof is a straightforward adaptation of \cite[Chapter 1, Theorem 11]{MR589888}, but we include it for completeness.
In the following, $r_1,r_2,r_3$ play the roles of $R,2R,2\e R$ in \cite{MR589888}. Without loss of generality we may and will assume that $\zeta_2=0$ and that $f(0)=1$; otherwise we could replace $f(z)$ by $\frac{f(z+\zeta_2)}{f(\zeta_2)}$. Consider the function 
\begin{align*}
\varphi(z):=\frac{(-r_2)^n}{z_1 z_2\ldots z_n}\prod_{k=1}^n\frac{r_2(z-z_k)}{r_2^2-\overline{z_k}z}.
\end{align*}
We recall that $z_1,z_2,\ldots, z_n$ are the zeros of $f$ in $D(0,r_2)$. Observe that $\varphi(0)=1$ and 
\begin{align*}
\varphi(r_2\e^{\I\theta})=\frac{r_2^n}{|z_1 z_2\ldots z_n|}
\end{align*}
for $\theta\in\R$. The function 
\begin{align*}
\Psi(z):=\frac{f(z)}{\varphi(z)}
\end{align*}
has no zeros in $D(0,r_2)$ and satisfies $\psi(0)=1$;  therefore, by Carath\'eodory's theorem \cite[Theorem 9]{MR589888}, for $|z|\leq r_1$,
\begin{align*}
\log|\psi(z)|&\geq -\frac{2r_1}{r_2-r_1}\big(\log\max_{|z|=r_2}|f(z)|+\log\frac{r_2^n}{|z_1 z_2\ldots z_n|}\big)\\
&\geq  -\frac{2r_1}{r_2-r_1}\log\max_{|z|=r_2}|f(z)|.
\end{align*}
To estimate $\varphi$ from below for $|z|\leq r_1$ we use
\begin{align*}
\prod_{k=1}^n|r_2^2-\overline{z_k}z|&<(2r_2^2)^n,\\
\prod_{k=1}^n|r_2(z-z_k)|&>\left(\frac{\delta r_2}{\e}\right)^nr_2^n,\quad z\notin Z_{f,2\delta r_2,U_2}.
\end{align*} 
The second inequality follows from Cartan's estimate \cite[Theorem 10]{MR589888}. We thus obtain the lower bound
\begin{align*}
|\varphi(z)|>(2r_2^2)^{-n}\left(\frac{\delta r_2}{\e}\right)^n\frac{r_2^{2n}}{|z_1 z_2\ldots z_n|}>\left(\frac{\delta}{2\e}\right)^n,\quad z\notin Z_{f,2\delta r_2,U_2}.
\end{align*}
By Jensen's formula \cite[Lemma 4]{MR589888}, since $f(0)=1$,
\begin{align*}
n\leq \log^{-1}\big(\frac{r_3}{r_2}\big)\log\max_{|z|=r_3}|f(z)|,
\end{align*}
and consequently
\begin{align*}
\log|\varphi(z)|>\log^{-1}\big(\frac{r_3}{r_2}\big)\log\max_{|z|=r_3}|f(z)|\log \left(\frac{\delta}{2\e}\right),\quad z\notin Z_{f,2\delta r_2,U_2}.
\end{align*}
Together with the lower bound for $\log|\psi|$ this leads to the claimed estimate, upon redefining $\delta$ and absorbing an error into the constant $C$.
\end{proof}

Next we state a version of Lemma \ref{lemma Levin} for ``wedges" of the form
\begin{align}\label{def. wedge}
W(\varphi,\theta;r,R):=\set{z\in\C\setminus[0,\infty)}{\arg(z)\in(2\varphi,2\theta),\, |z|\in (r^2,R^2)}
\end{align}
and $W(\varphi,\theta;R):=W(\varphi,\theta;R,\infty)$, where $0\leq\varphi<\theta\leq\pi$. In the following we fix $0\leq \varphi_3<\varphi_2<\varphi_1<\theta_1<\theta_2<\theta_3\leq\pi$ and $0<r_3<r_2<r_1<R_1<R_2$, and define 
\begin{align*}
U_1:=W(\varphi_1,\theta_1;r_1,R_1),\quad
U_2:=W(\varphi_2,\theta_2;r_2,R_2),\quad
U_3:=W(\varphi_3,\theta_3;r_3).
\end{align*}

\begin{lemma}\label{lemma wedge}
Assume $f$ is a bounded holomorphic function on $U_3$ and that 
\begin{align}\label{condition lemma wedge}
r_3\ll r_2,\!\!\!\!\quad\frac{\rd(\partial U_2,\partial U_3)}{\rd(\partial U_1,\partial U_3)}\ll  \left(\frac{r_2}{R_2}\right)^{\frac{2\pi}{\theta_3-\varphi_3}+2},\!\!\!\!\quad \frac{\rd(\partial U_1,\partial U_3)}{(\theta_3-\varphi_3)R_2^2} \left(\frac{r_2}{R_2}\right)^{\frac{2\pi}{\theta_3-\varphi_3}}\ll 1.
\end{align}
Then there exists an absolute constant $C$ such that for any $\zeta_2\in U_2$ and any sufficiently small $\delta>0$,
\begin{align}\label{eq. lemma wedge}
\log|f(z)|\geq -C\frac{R_2^2}{\rd(\partial U_2,\partial U_3)} \left(\frac{R_2}{r_2}\right)^{\frac{2\pi}{\theta_3-\varphi_3}}\log\frac{1}{\delta}\big(\log\max_{z\in U_3}|f(z)|-\log|f(\zeta_2)|\big)
\end{align}
for all $z\in U_1\setminus Z_{f,\delta,U_2}$. 
\end{lemma}

\begin{remark}
The constant $C$ only depends on the implicit constants in~\eqref{condition lemma wedge}. Hence, one can optimize the inequality with respect to $\zeta_2$, subject to the conditions above.
\end{remark}

\begin{proof}
We map $U_3$ conformally onto the unit disk, using a composition of the following conformal maps (where, by abuse of notation, we denote the variable and the map by the same letter):
\begin{itemize}
\item[i)] $U_3\to\kappa(U_3)\subset \mathbb{H}$, $\kappa(z):=\sqrt{z}$;
\item[ii)] $\kappa(U_3)\to S:=\set{\sigma\in \C}{0<\im\sigma<\pi,\,\re\sigma>0}$, 
\begin{align*}
\sigma(\kappa):=\log(\e^{-\frac{\I\pi\varphi_3}{\theta_3-\varphi_3}}(\kappa/r_3)^{\frac{\pi}{\theta_3-\varphi_3}}),
\end{align*}
where we select the principal branch of the logarithm on $\C\setminus[0,-\I\,\infty)$;
\item[iii)] The Schwarz-Christoffel transformation $S\to \mathbb{H}$, $\tau(\sigma):=\cosh(\sigma)$. 
\item[iv)] The Möbius transformation $\mathbb{H}\to D(0,1)$, $w(\tau):=\frac{\tau-\tau_2}{\tau+\tau_2}$, where $\tau_2:=\tau(\sigma(\sqrt{\zeta_2}))$.
\end{itemize} 
The choice of $\tau_2$ has been made in such a way that $w(z)=0$ if $z=\zeta_2$. Here we again abuse notation and write  $w(z)=w(\tau(\sigma(\sqrt{z})))$. Note that
\begin{align}
\tau(\sigma(\kappa))=\frac{1}{2}(\alpha(\kappa)+\alpha(\kappa)^{-1}),\quad
\alpha(\kappa):=\e^{-\frac{\I\pi\varphi_3}{\theta_3-\varphi_3}}(\kappa/r_3)^{\frac{\pi}{\theta_3-\varphi_3}}.
\end{align}
By distortion bounds \cite[Cor. 1.4]{MR1217706},
\begin{align}\label{Koebe distortion}
 \left|\frac{\rd w(z)}{\rd z}\right|\rd(z,U_3)\leq 1-|w(z)|^2\leq 4  \left|\frac{\rd w(z)}{\rd z}\right|\rd(z,U_3).
\end{align}
We compute the differential of $w$ at $z\in U_2$ by the chain rule,
\begin{align*}
&\left|\frac{\rd w}{\rd z}\right|=\left|\frac{\rd w}{\rd\tau}\frac{\rd \tau}{\rd\sigma}\frac{\rd \sigma}{\rd\kappa}\frac{\rd\kappa}{\rd z}\right|=\frac{\pi|\tau_2||\sinh(\sigma)|}{(\theta_3-\varphi_3)|z|(\tau+\tau_2)^2}\\
&=\frac{\pi}{(\theta_3-\varphi_3)|z|}\frac{|\alpha(\kappa_2)+\alpha(\kappa_2)^{-1}||\alpha(\kappa)+\alpha(\kappa)^{-1}|}{|\alpha(\kappa_2)+\alpha(\kappa_2)^{-1}+\alpha(\kappa)+\alpha(\kappa)^{-1}|^2}.
\end{align*}
Since $|\kappa|\geq r_2\gg r_3$, we have $|\alpha(\kappa)|\gg 1$, whence
\begin{align*}
\frac{|\alpha(\kappa_2)+\alpha(\kappa_2)^{-1}||\alpha(\kappa)+\alpha(\kappa)^{-1}|}{|\alpha(\kappa_2)+\alpha(\kappa_2)^{-1}+\alpha(\kappa)+\alpha(\kappa)^{-1}|^2}
\asymp \frac{|\alpha(\kappa_2)||\alpha(\kappa)|}{(|\alpha(\kappa_2)|+|\alpha(\kappa)|)^2},
\end{align*}
which, in view of $|\alpha(\kappa)|=(|\kappa|/r_3)^{\frac{\pi}{\theta_3-\varphi_3}}$, leads to
\begin{align*}
\frac{1}{(\theta_3-\varphi_3)R_2^2} \left(\frac{r_2}{R_2}\right)^{\frac{2\pi}{\theta_3-\varphi_3}}\lesssim\left|\frac{\rd w}{\rd z}\right|\lesssim 
\frac{1}{(\theta_3-\varphi_3)r_2^2} 
\end{align*}
for $z\in U_2$. Denoting the numbers $r_j$ in Lemma \ref{lemma Levin} by $\rho_j$ instead (with $\rho_3=1$), we then find, using \eqref{Koebe distortion},
\begin{align*}
1-\rho_2&\gtrsim \frac{\rd(\partial U_2,\partial U_3)}{(\theta_3-\varphi_3)R_2^2} \left(\frac{r_2}{R_2}\right)^{\frac{2\pi}{\theta_3-\varphi_3}},\\
\rho_2-\rho_1&\gtrsim \frac{\rd(\partial U_1,\partial U_3)}{(\theta_3-\varphi_3)R_2^2} \left(\frac{r_2}{R_2}\right)^{\frac{2\pi}{\theta_3-\varphi_3}}-\frac{\rd(\partial U_2,\partial U_3)}{(\theta_3-\varphi_3)r_2^2}
\gtrsim \frac{\rd(\partial U_1,\partial U_3)}{(\theta_3-\varphi_3)R_2^2} \left(\frac{r_2}{R_2}\right)^{\frac{2\pi}{\theta_3-\varphi_3}}, 
\end{align*}
where in the second line we used the triangle inequality and the second inequality in \eqref{condition lemma wedge}. By the third inequality in \eqref{condition lemma wedge} we can Taylor expand 
\begin{align*}
\log(\frac{1}{\rho_2})=-\log(1-(1-\rho_2))\asymp  1-\rho_2.
\end{align*}
Lemma \ref{lemma Levin} now yields the claim.  
\end{proof}

\subsection{Distribution function}\label{subsection Distribution function}

For $s>0$, we define
\begin{align*}
h_L(s)=|\set{k\in [N]}{\eta_0 L_k\leq  1/s}|\in \Z_+,
\end{align*}
where $\eta_0^{-1}$ is an arbitrary length scale. Note that $h_L$ is decreasing and tends to infinity as $s\to 0$.
In fact, $h_L$ is the distribution function of the sequence $(\eta_0L_k)^{-1}$. Since we assume that $L_k$ is increasing, we also have
\begin{align*}
h_L(s)=\min\set{k\in\Z_+}{\eta_0 L_{k+1}> 1/s}.
\end{align*}
We will show that, under the assumption
\begin{align}\label{assumption on h_L}
\exists\lambda\in (0,1)\quad \mbox{such that}\quad\limsup_{s\to 0+}\frac{h_L(\lambda s)}{\e\, h_L(s)}<1,
\end{align}
the potential $V(L)$ is strongly separating in the sense of Definition \ref{def. sparse}. 

\begin{proposition}\label{lemma distribution function}
Assume \eqref{assumption on h_L}. Then
\begin{align}\label{bound of lemma distribution function}
\mathrm{sep}(L,\eta)\lesssim\exp(-\eta L_1)\langle h_{L}(\eta/\eta_0)\rangle.
\end{align}
\end{proposition}

In particular, this implies that the examples in Subsection  \ref{subsection Separating and sparse potentials} are strongly separating:
\begin{itemize}
\item[a)] If $\eta_0L_k\gtrsim k^{\alpha}$ for $\alpha>0$, then $h_L(s)\lesssim s^{-1/\alpha}$. 
\item[b)] If $\eta_0L_k\gtrsim\exp(k)$, then $h_L(s)\lesssim\log(1/s)$.
\item[c)] If $\eta_0L_k\gtrsim\exp(\exp(k))$, then $h_L(s)\lesssim\log\log(1/s)$.
\end{itemize}

\begin{lemma}\label{lemma s vs delta s}
Assume \eqref{assumption on h_L}. Then for any $\delta>0$ and for all $s>0$,
\begin{align*}
\langle h_L(\delta s)\rangle\lesssim_{\delta} \langle h_L(s)\rangle.
\end{align*}
\end{lemma}

\begin{proof}
We may restrict our attention to the case $\delta<1$ as the case $\delta\geq 1$ is trivial. By \eqref{assumption on h_L} there exist $\lambda\in (0,1)$ and $s_0>0$ such that
\begin{align}\label{h_L(s) vs h_L(delta s)}
h_L(\lambda s)<\e\, h_L(s)
\end{align}
holds for all $s\in (0,s_0]$. Now let $n$ be the smallest integer such that $\lambda^n\leq \delta$. Iterating \eqref{h_L(s) vs h_L(delta s)} $n$ times, we get
\begin{align*}
h_L(\delta s)<\e^n h_L(s)
\end{align*}
for all $s\in (0,s_0]$. For $s>s_0$, the inequality holds trivially. 
\end{proof}

\begin{proof}[Proof of Proposition \ref{lemma distribution function}] Without loss of generality we may assume that $\eta=\eta_0$.
We first consider the case $\eta L_1\leq 1$, and hence $h_L(1)\geq 1$. 
Then 
\begin{align*}
\sum_{k=1}^{\infty}\exp(-\eta L_k)\leq {h_L(1)}\exp(-\eta L_1)+\sum_{k=h_L(1)}^{\infty}\exp(-\eta L_k).
\end{align*}
It remains to show that the second term is bounded by the right hand side of \eqref{bound of lemma distribution function}. To this end, we decompose the sum into dyadic intervals $I_j=[h_L(2^{-j}),h_L(2^{-j-1})]$, $j\in \Z_+$. Then 
\begin{align*}
\sum_{k\in I_j}\exp(-\eta L_k)\leq \exp(-2^j)h_L(2^{-j-1}).
\end{align*}
Summing over $j$ and using Cauchy's condensation test yields
\begin{align}\label{Cauchy condensation}
\sum_{k=h_L(1)}^{\infty}\exp(-\eta L_k)\lesssim \sum_{n=1}^{\infty}\exp(-n)h_L(\tfrac{1}{2n}).
\end{align}
By the quotient test, the series converges provided that
\begin{align}\label{quotient test}
\limsup_{n\to\infty}\frac{h_L(\lambda_n\tfrac{1}{2n})}{\e\,h_L(\tfrac{1}{2n})}<1,
\end{align}
where $\lambda_n=\tfrac{n}{n+1}$.
But this follows from assumption \eqref{assumption on h_L}. Indeed, since $\lambda_n\to 1$, we have $\lambda< \lambda_n$ for large $n$ and hence $h_L(\lambda_n\tfrac{1}{2n})\leq h_L(\lambda\tfrac{1}{2n})$. The series~\eqref{Cauchy condensation} is thus bounded by $\langle h_L(1)\rangle$, where we have used Lemma \ref{lemma s vs delta s} with $\delta=1/2$. This proves \eqref{bound of lemma distribution function} in the case $\eta L_1\leq 1$. The case $\eta L_1>1$ is similar, but \eqref{Cauchy condensation} is bounded by $\exp(-n_0)h_L(\tfrac{1}{2n_0})$, where $n_0$ is the least integer such that $h_L(\tfrac{1}{2n_0})\geq 1$, i.e.\ $n_0=\lceil\tfrac{\eta L_1}{2}\rceil$. Another application of Lemma \ref{lemma s vs delta s} completes the proof.
\end{proof}

\bibliographystyle{abbrv}

\end{document}